\documentclass{gtart}
\usepackage{t1enc}
\usepackage[latin1]{inputenc}
\usepackage[english]{babel}
\usepackage{amssymb}
\usepackage{amsmath}
\usepackage{graphics}
\usepackage{hyperref}

\usepackage{amsthm, amssymb}
\usepackage{amsfonts}
\usepackage{epsfig,multicol}   

\usepackage{mathrsfs}
\newtheorem{theorem}[subsection]{Theorem}

\newtheorem{lemma}[subsection]{Lemma}
\newtheorem{corollary}[subsection]{Corollary}
\newtheorem{definition}[subsection]{Definition}
\newtheorem{example}[subsection]{Example}
\newtheorem{remark}[subsection]{Remark}

\newtheorem{conjecture}[subsection]{Conjecture}
\newcommand{\arxiv}[1]{\href{http://arxiv.org/abs/#1}{\tt arXiv:\nolinkurl{#1}}}
\newcommand{\doi}[1]{\href{http://dx.doi.org/#1}{{\tt DOI:#1}}}

\newcommand{\googlebooks}[1]{(preview at \href{http://books.google.com/books?id=#1}{google books})}

\def\<{\langle}
\def\>{\rangle}

\begin{document}

\def\hpic #1 #2 {\mbox{$\begin{array}[c]{l} \epsfig{file=#1,height=#2}
\end{array}$}}
 
\def\vpic #1 #2 {\mbox{$\begin{array}[c]{l} \epsfig{file=#1,width=#2}
\end{array}$}}

\title{Quantum subgroups of the Haagerup fusion categories}
\author{Pinhas~Grossman}
\address{
}%
\email{pinhas@impa.br}

\author{Noah~Snyder}
\address{
}%
\email{nsnyder@math.columbia.edu}

\address{%
\rm URLs:\stdspace \tt \url{http://math.columbia.edu/~nsnyder} \\ \url{http://w3.impa.br/~pinhas}}


\begin{abstract}
We answer three related questions concerning the Haagerup subfactor and its even parts, the Haagerup fusion categories.  Namely we find all simple module categories over each of the Haagerup fusion categories (in other words, we find the ``quantum subgroups'' in the sense of Ocneanu), we find all subfactors whose principal even part is one of the Haagerup fusion categories, and we compute the Brauer-Picard groupoid of Morita equivalences of the Haagerup fusion categories.  In addition to the two even parts of the Haagerup subfactor, there is exactly one more fusion category which is Morita equivalent to each of them.  This third fusion category has six simple objects and the same fusion rules as one of the even parts of the Haagerup subfactor, but has not previously appeared in the literature.  We also find the full lattice of intermediate subfactors for every subfactor whose even part is one of these three fusion categories, and we discuss how our results generalize to Izumi subfactors.\end{abstract}

\maketitle

\section{Introduction}
The Haagerup subfactor is a finite-depth subfactor with index $\frac{5+\sqrt{13}}{2} $; this is the smallest index above $4$ for any finite depth subfactor.  Its even parts are two fusion categories.  We will call the fusion category with four simple objects $\mathscr{H}_1$ and the one with six simple objects $\mathscr{H}_2$.   


Given a fusion category $\mathscr{C}$, one important structural question is to understand all quantum subgroups of $\mathscr{C}$ in the sense of Ocneanu \cite{MR1907188}.  That is we wish to understand all simple module categories over  $\mathscr{C}$ (by simple we mean semisimple and indecomposable).  The reason for the name ``quantum subgroup'' is that when $\mathscr{C}$ is the category of $G$-modules for a finite group $G$ then the simple module categories correspond to  the subgroups $H \subset G$ together with some additional cohomological data \cite{MR1976233}.  When $\mathscr{C}$ is a fusion category coming from quantum $\mathfrak{su}_2$ at a root of unity, then the quantum subgroups are given by the ADE Dynkin diagrams.  (See \cite{MR996454, OcnLect, MR1777347} for this result in subfactor language, and \cite{MR1936496, MR1976459, MR2046203} for the translation of these results into the language of fusion categories and module categories.)  Ocneanu has announced the classification of quantum subgroups of the fusion categories coming from quantum $\mathfrak{su}_3$ and $\mathfrak{su}_4$ \cite{MR1907188} (see \cite{MR2545609,MR2553429} for details in the $\mathfrak{su}_3$ case).  In this paper we find all quantum subgroups of $\mathscr{H}_1$ and $\mathscr{H}_2$.  One might think of the results of this paper as analogous to finding the subgroups of a sporadic finite simple group.

\begin{theorem}
There are exactly three quantum subgroups of each of $\mathscr{H}_1$ and $\mathscr{H}_2$.  The quantum subgroups of  $\mathscr{H}_1$ have the following graphs for fusion with the object of dimension $\frac{1+\sqrt{13}}{2}$.

\hpic{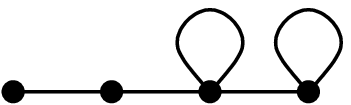} {.3in} \quad \hpic{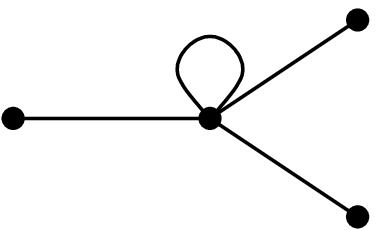} {.6in} \quad \hpic{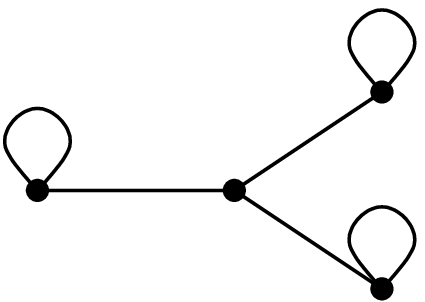} {.6in}

The quantum subgroups of $\mathscr{H}_2$ have the following graphs for fusion with one of the objects of dimension $\frac{3+\sqrt{13}}{2}$.

\hpic{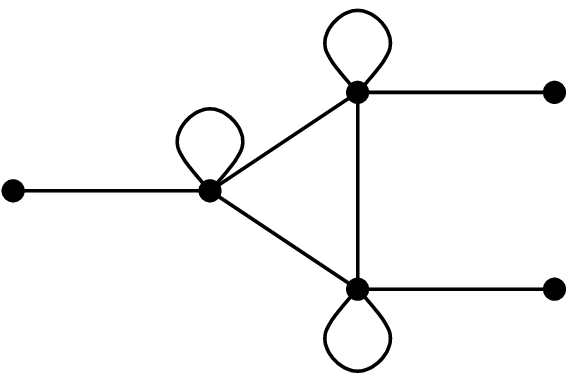} {.6in} \quad
\hpic{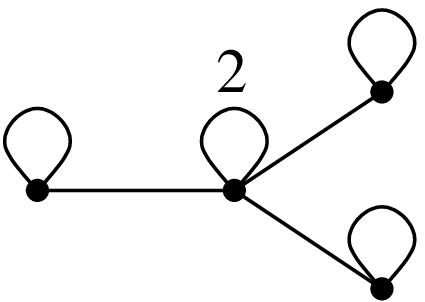} {.6in} \quad
\hpic{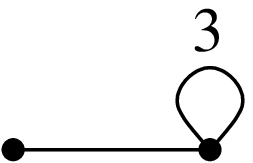} {.4in}
\end{theorem}

The graphs in the above theorem are analogous to the ADE Dynkin diagrams for quantum subgroups of quantum $\mathfrak{su}_2$.

Understanding all quantum subgroups of $\mathscr{C}$ also allows us to answer several other important structural questions about  $\mathscr{C}$.  A subfactor whose principal even part is $\mathscr{C}$ is roughly the same thing as a simple algebra object in $\mathscr{C}$ (see Section \ref{sec:subfactorvsalgebra} and \cite{MR1257245, MR1444286, MR1966524, MR2075605}).  All simple algebra objects in $\mathscr{C}$ can be realized as the internal endomorphisms of a simple object in some module category over $\mathscr{C}$ \cite{MR1976459}.  Hence we can use our classification of quantum subgroups to describe all subfactors whose even parts are $\mathscr{H}_1$ or  $\mathscr{H}_2$ (this generalizes the GHJ subfactors \cite{MR999799} constructed from the quantum subgroups of quantum $\mathfrak{su}_2$).

\begin{theorem}
There are exactly $7$ subfactors of the hyperfinite $II_1$ factor whose principal even part is $\mathscr{H}_1$.  These subfactors have the following principal graphs, dual principal graphs, and indices.

\begin{tabular}{ccc}
\hpic{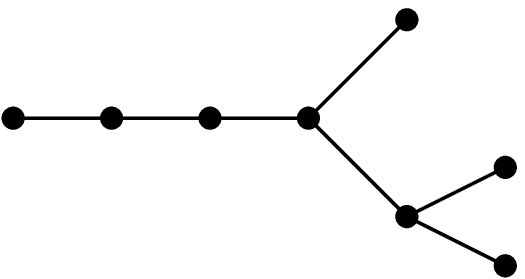} {0.6in} & \hpic{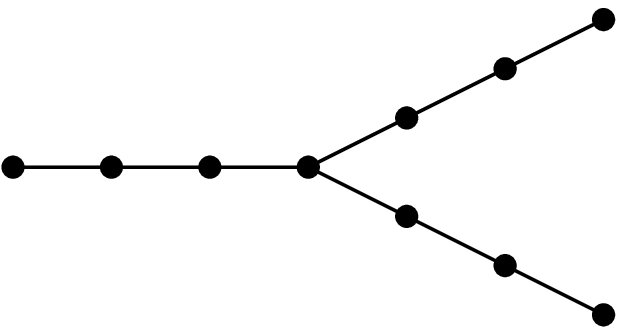} {0.6in} &  $\frac{5+\sqrt{13}} {2}$\\
\hpic{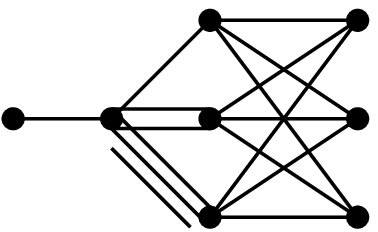} {0.6in} & \hpic{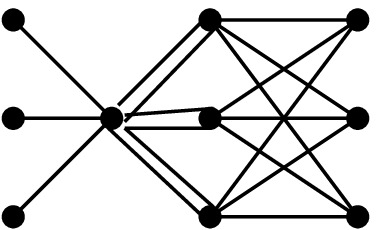} {0.6in} &  $12+3\sqrt{13}$\\
\hpic{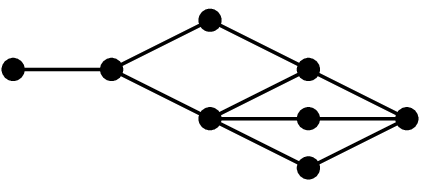} {0.6in} & \hpic{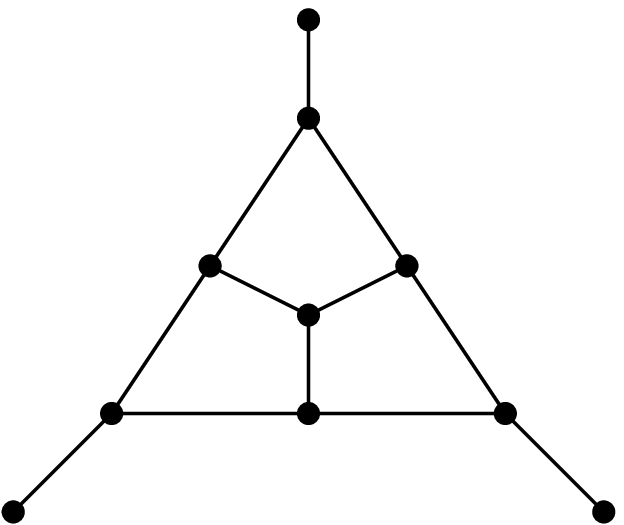} {0.6in} &  $4+\sqrt{13}$\\
\hpic{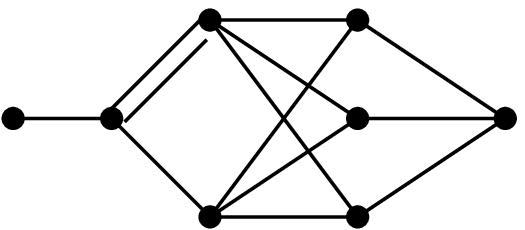} {0.6in} & \hpic{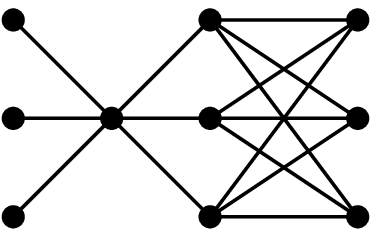} {0.6in} &  $\frac{15+3\sqrt{13}}{2}$\\
\hpic{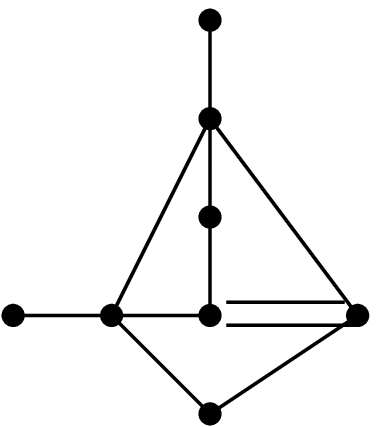} {0.6in} & \hpic{hu4} {0.6in} &  $\frac{11+3\sqrt{13}}{2}$\\
\hpic{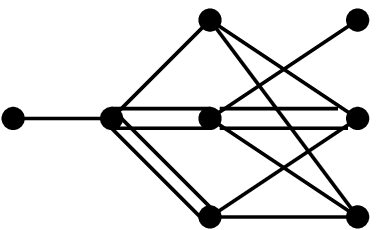} {0.6in} & \hpic{hu6} {0.6in} &  $\frac{19+5\sqrt{13}}{2}$\\
\hpic{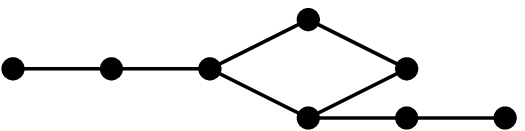} {0.3in} & \hpic{hu7} {0.3in} &  $\frac{7+\sqrt{13}}{2}$\\
\end{tabular}
\end{theorem} 

\begin{theorem}
There are exactly $4$ subfactors of the hyperfinite $II_1$ factor whose principal even part is $\mathscr{H}_2$.  These subfactors have the following principal graphs, dual principal graphs, and indices.

\begin{tabular}{ccc}
\hpic{hu8} {0.6in} & \hpic{hu1} {0.6in} &  $\frac{5+\sqrt{13}} {2}$\\
\hpic{hu9} {0.6in} & \hpic{hu2} {0.6in} &  $12+3\sqrt{13}$\\
\hpic{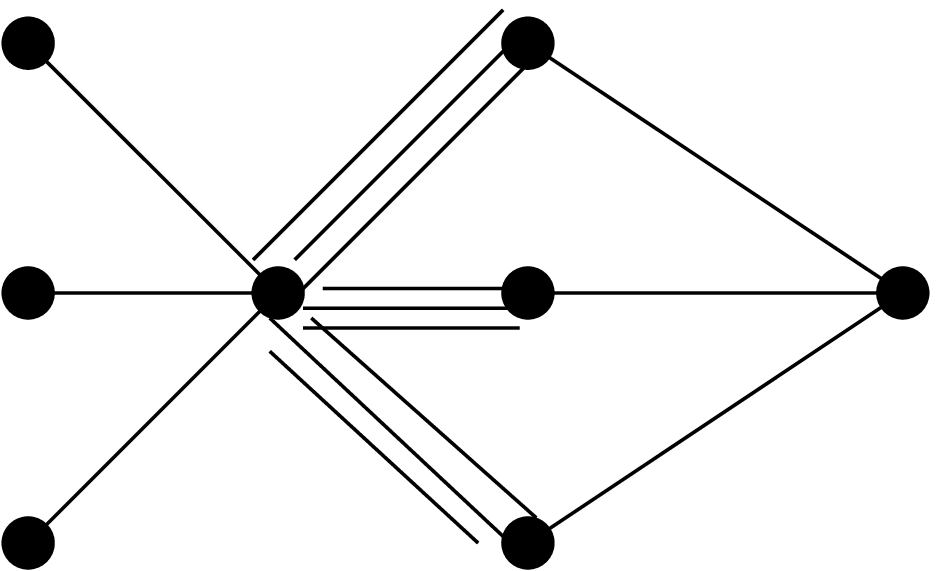} {0.6in} & \hpic{hu13} {0.6in} &  $\frac{33+9\sqrt{13}}{2}$\\
\end{tabular}

\begin{tabular}{ccc}\hpic{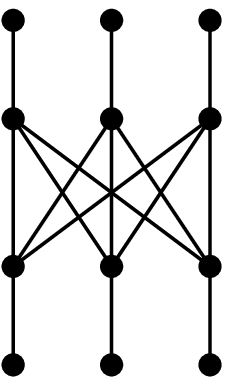} {0.6in} & \hpic{hu11} {0.6in} &  $\frac{11+3\sqrt{13}}{2}$\\
\end{tabular}
\end{theorem}

Note that for several of the subfactors in the above lists the dual subfactor does not appear on either list.  This is because for these subfactors the dual even part is not $\mathscr{H}_1$ nor $\mathscr{H}_2$.  Instead it is a new fusion category, which we call $\mathscr{H}_3$.

\begin{theorem}
The (higher) Morita equivalence class of $\mathscr{H}_1$ consists of exactly three fusion categories: $\mathscr{H}_1$, $\mathscr{H}_2$, and $\mathscr{H}_3$.  The fusion category $\mathscr{H}_3$ has six objects, three of dimension $1$ and three of dimension $\frac{3+\sqrt{13}}{2}$, and the same fusion ring as $\mathscr{H}_2$.  The fusion category $\mathscr{H}_3$ has exactly three quantum subgroups, the fusion graphs for these quantum subgroups (with respect to any of the $\frac{3+\sqrt{13}}{2}$ dimensional objects) are: 

\hpic{m4} {.6in} \quad
\hpic{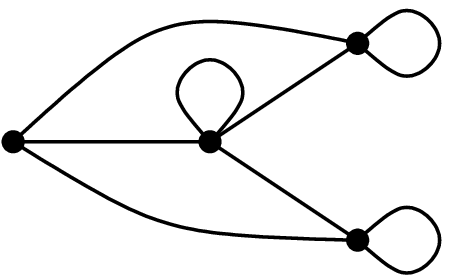} {.6in} \quad
\hpic{m6} {.4in}

Furthermore, $\mathscr{H}_3$  appears as the even part of exactly four subfactors whose principal graphs, dual principal graphs, and indices are listed below.

\begin{tabular}{ccc}
\hpic{hu10} {0.6in} & \hpic{hu3} {0.6in} &  $4+\sqrt{13}$\\
\hpic{hu12} {0.6in} & \hpic{hu5} {0.6in} &  $\frac{15+3\sqrt{13}}{2}$\\
\hpic{hu13} {0.6in} & \hpic{hu13} {0.6in} &  $\frac{33+9\sqrt{13}}{2}$\\
\hpic{hu11} {0.6in} & \hpic{hu11} {0.6in} &  $\frac{11+3\sqrt{13}}{2}$\\
\end{tabular}
\end{theorem}

The fusion category $\mathscr{H}_3$ mentioned above is new.  It can be described most succinctly as the category of $(1+\alpha+\alpha^2)$-bimodule objects in $\mathscr{H}_2$ where $\alpha$ and $\alpha^2$ are the nontrivial invertible objects in $\mathscr{H}_2$.  It can also be described via an intermediate subfactor construction.  Namely, consider the Haagerup subfactor, and then look at the reduced subfactor constructed from the middle vertex of the principal graph.  This subfactor has three invertible objects at depth $2$ and thus has an intermediate of index $3$.  The other intermediate subfactor has $\mathscr{H}_3$ as one of its even parts.

Note that although the fusion categories $\mathscr{H}_2$ and $\mathscr{H}_3$ have the same Grothendieck rings, they are not the only fusion categories with this Grothendieck ring.  In particular both fusion categories have non-unitary Galois conjugates.  Furthermore according to \cite{1006.1326} there may be unitary fusion categories which differ from $\mathscr{H}_2$ by twisting by a class in $H^3(\mathbb{Z}/3\mathbb{Z}, \mathbb{C}^*)$ which are not Morita equivalent to $\mathscr{H}_2$ (since their centers would have different modular invariants).  The classification of all fusion categories whose Grothendieck ring agrees with that of $\mathscr{H}_2$ remains an interesting open question (which was suggested to us by Pavel Etingof).


The Brauer-Picard groupoid of a fusion category $\mathscr{C}$ has points for every fusion category $\mathscr{D}$ which is Morita equivalent to $\mathscr{C}$ and an arrow for every Morita equivalence between $\mathscr{C}$ and $\mathscr{D}$ (up to equivalence of bimodule categories).  The Brauer-Picard groupoid is important in understanding the extension theory of $\mathscr{C}$ \cite{0909.3140}, and is essentially the same as Ocneanu's notion of ``maximal atlas.''  Morita equivalences between $\mathscr{C}$ and $\mathscr{D}$ correspond to a choice of module category over $\mathscr{C}$ and a choice of isomorphism between $\mathscr{D}$ and $\mathscr{C}_\mathscr{M}^*$ up to inner automorphism of $\mathscr{D}$. Thus understanding Morita equivalences requires understanding the module categories and the outer automorphisms of $\mathscr{C}$ and its Morita equivalent categories.

\begin{theorem}
The Brauer-Picard groupoid of $\mathscr{H}_1$ has three objects $\mathscr{H}_1$, $\mathscr{H}_2$, $\mathscr{H}_3$, and between any two (not necessarily distinct) of these there is exactly one Morita equivalence.  In particular, none of the $\mathscr{H}_i$ has an outer automorphism.
\end{theorem}

The main technique of this paper is to move back-and-forth between algebra objects and module categories to exploit the richer combinatorial structure of the former and the richer algebraic structure of the latter.

In Section 2 we recall some background information on fusion categories, module categories, algebra objects, and subfactors. In Section 3 we find all fusion categories Morita equivalent to the Haagerup fusion categories, and classify all algebra objects in and module categories over these fusion categories. In Section 4 we find the full intermediate subfactor lattices for all subfactors arising from the quantum subgroups of the Haagerup fusion categories. In Section 5 we show that the outer automorphisms groups of the Haagerup fusion categories are trivial and compute the Brauer-Picard groupoid. In Section 6 we discuss how our results generalize to the Izumi subfactors, in particular the Izumi subfactor corresponding to $\mathbb{Z} / 5\mathbb{Z} $.

We would like to thank David Penneys who pointed out to us that we were both working on this problem.  Noah Snyder would like to thank Dmitri Nikshych who suggested this question at the Shanks workshop on Subfactors and Fusion Categories at Vanderbilt, and Dietmar Bisch and Richard Burstein for hosting that conference.  Pinhas Grossman would like to thank David Evans for helpful conversations, and Noah Snyder would like to thank David Jordan and Emily Peters.  Noah Snyder was supported by an NSF Postdoctoral Fellowship at Columbia University.  Pinhas Grossman was supported by a Marie Curie fellowship from the EU Noncommutative Geometry Network at Cardiff University, and later by a fellowship at IMPA in Brazil. He was also partially supported by NSF grant DMS-0801235. 

\section{Background}

\subsection{Translating between fusion category language and subfactor language}

The goal of this subsection is to explain how to translate between fusion category language and subfactor language.  It is a cheat sheet for the rest of the background section, and we hope that it will enable people who are only familiar with one of the languages to understand the rest of the paper.  The key point is that if you have a subfactor $N \subset M$ then you have a tensor category of all $N$-$N$ bimodules together with an algebra $M$ which can be thought of as an $N$-$N$ bimodule.  Thus subfactors roughly correspond to algebra objects in tensor categories.

The table below should be taken with three caveats, two of them technical and one of them important.  First, since factors are $C^*$ algebras, the corresponding tensor categories are always unitary and one may need to impose certain compatibility conditions with the $*$-structure.  Second, the word ``fusion'' means ``finite depth'' in the context of subfactors, so for infinite depth subfactors you should relax the adjective ``fusion.''  The important caveat will be explained below.

\begin{center}\begin{tabular}{l|l}
Subfactors & Fusion categories \\
\hline
\hline
$N \subset M$ & A fusion category $\mathscr{C}$ with an algebra object $A$ \\
\hline
The standard invariant & The $2$-category consisting of all $1$-$1$, $1$-$A$,  $A$-$1$, \\& and $A$-$A$ bimodule objects in $\mathscr{C}$ together with \\& a choice of $1$-morphism $A$ as an $A$-$1$ bimodule \\
\hline
Principal even part & $\mathscr{C}$ \\
\hline
Odd part & The collection of $A$-modules (called $\mathscr{M}$) \\ & as a right module category over $\mathscr{C}$ \\
\hline
Dual even part & The dual fusion category $\mathscr{C}_\mathscr{M}^*$, or equivalently \\& the fusion category of $A$-$A$ bimodules. \\
\hline
The principal graphs & The fusion graphs for tensoring with $A$ \\
\hline
$Q$-system & Algebra object \\
\hline
Tensoring odd objects & Internal Hom \\
\hline
\end{tabular} \end{center}

\vspace{.1in}

Now for the important caveat.  Notice that subfactors always correspond to a category theoretic construction {\em together with} a choice of object $A$.  Furthermore, notice that the category of {\em all} $N$-$N$ bimodules is quite large and unwieldy.  Thus, in subfactor theory, one always restricts to the subcategory tensor generated by $A$.  However, once you fix a fusion subcategory inside all $N$-$N$ bimodules (as we do throughout this paper), it then becomes unnatural to look only at the subcategory tensor generated by $A$.  Much of the power of this paper comes from using constructions which are natural on the fusion category side, but somewhat unnatural on the subfactor side because the $Q$-system doesn't tensor generate.  For example, if you look at the algebra object $1$ inside $\mathscr{C}$, then the category of all $1$-$1$ bimodules is $\mathscr{C}$ itself, but the corresponding subfactor is trivial.   We will often prove the uniqueness of a nontrivial $Q$-system by proving the uniqueness of a simpler $Q$-system which doesn't tensor generate.  To our knowledge, these arguments do not correspond to any arguments already appearing in the subfactor literature.

\subsection{Fusion categories, module categories, and bimodule categories}

In this subsection we sketch the definitions of fusion categories, module categories, bimodule categories, outer automorphisms and the Brauer-Picard groupoid.  For more details see \cite{MR2183279, MR1976459, 0909.3140}.

\begin{definition} \cite{MR2183279}
A fusion category over  $\mathbb{C}$ is a $\mathbb{C}$-linear rigid semisimple monoidal category with finitely many simple objects (up to isomorphism) and finite-dimensional morphism spaces, such that the identity object is simple.  
\end{definition}

Recall that a monoidal functor $\mathscr{F}:\mathscr{C} \rightarrow \mathscr{D}$ is a pair a functor $\mathscr{F}$ together with a binatural transformation $\sigma_{X,Y}: \mathscr{F}(X \otimes Y) \rightarrow  \mathscr{F}(X) \otimes  \mathscr{F}(Y)$ satisfying a certain naturality condition with the associator.  Somewhat nonstandardly we call an invertible monoidal functor an isomorphism (or an automorphism if it is an endofunctor) rather than an ``equivalence."  The main reason for this is that we want to avoid confusion with {\bf Morita} equivalence, also it is harmless since the usual definition of ``isomorphism" in category theory is well-known to be useless.  An automorphism is called inner if it is given by conjugation by an invertible object.  The group of outer automorphisms is the quotient of the group of all automorphisms by the subgroup of inner automorphisms.

The Grothendieck ring or fusion ring of a fusion category is the Grothendieck group of the fusion category (that is formal differences of objects modulo the natural relations) with multiplication by tensor product.

\begin{definition} \cite{MR2183279}
A dimension function is an assignment of a complex number to every object of $\mathbb{C}$ which is multiplicative under tensor product and additive under direct sum.  There is a unique dimension function which sends all non-zero objects to a positive real number, called the Frobenius-Perron dimension.  It assigns to each object $X$ the Frobenius-Perron eigenvalue of the matrix of left multiplication by $[X]$ in the Grothendieck ring.
\end{definition}

\begin{definition} \cite[Def. 6]{MR1976459}
A left module category over $\mathscr{C}$ is a $\mathbb{C}$-linear category $\mathscr{M}$ together with a biexact bifunctor $\otimes: \mathscr{C} \times \mathscr{M} \rightarrow \mathscr{M}$ and natural associtivity and unit isomorphisms satisfying certain coherence conditions.  A right module category over $\mathscr{C}$ and a $\mathscr{C}$-$\mathscr{D}$ bimodule category are defined similarly.
\end{definition}

As with fusion categories, a functor of module categories is a functor of the underlying category together with a binatural transformation $c_{X,M}: F(X \otimes M) \rightarrow F(X) \otimes F(M)$ satisfying certain compatibility relations.  See \cite[Def. 7]{MR1976459}.

\begin{definition}
A module category is called indecomposable if it cannot be written as the direct sum of two module categories and simple if it is semisimple and indecomposable.
\end{definition}

\begin{definition} \cite[\S3.2]{MR1976459}
Let $\mathscr{M}$ be a semisimple module category over a fusion category $\mathscr{C}$.
Let $M_1$ and $M_2$ be two objects of $\mathscr{M}$.  Their internal Hom, $\underline{Hom}(M_1, M_2)$ is the (unique up to unique isomoprhism) object of $\mathscr{C}$ which represents the functor $X \mapsto Hom(X \otimes M_1, M_2)$.
\end{definition}

As you might expect given the name, you can compose internal Homs: $$\underline{Hom}(M_2 , M_3)\otimes \underline{Hom}(M_1, M_2) \rightarrow \underline{Hom}(M_1, M_3).$$

\begin{definition}
Given a semisimple fusion category $\mathscr{C}$ and a semisimple module category $\mathscr{M}$, we can define the Frobenius-Perron dimension of objects in $\mathscr{M}$ by $\dim (M) = \sqrt{\dim(\underline{Hom}(M,M))}$.
\end{definition}

One important notion in the theory of rings and modules is that of Morita equivalence.  Two rings are Morita equivalent if there is an invertible bimodule between them.  Since fusion categories categorify rings and module categories categorify modules we should have a categorified version of Morita equivalence of fusion categories.

\begin{definition}
A Morita equivalence between $\mathscr{C}$ and $\mathscr{D}$ is an invertible $\mathscr{C}$-$\mathscr{D}$ bimodule category, i.e. $\mathscr{C}$-$\mathscr{D}$ bimodule category $\mathscr{M}$ such that there exists $\mathscr{M}'$ a $\mathscr{D}$-$\mathscr{C}$ such that $\mathscr{M} \otimes_\mathscr{D} \mathscr{M}'$ is equivalent to $\mathscr{C}$ as a $\mathscr{C}$-$\mathscr{C}$ bimodule category and $\mathscr{M}' \otimes_\mathscr{C} \mathscr{M}$ is equivalent to $\mathscr{D}$ as a $\mathscr{D}$-$\mathscr{D}$ bimodule category.  See \cite{MR1966524,0909.3140} for more details.
\end{definition}

The collection of all Morita equivalences naturally forms a $3$-groupoid.  We can think of it as just a groupoid by modding out by equivalences of bimodule categories.

\begin{definition}
The Brauer-Picard groupoid of a fusion category $\mathscr{C}$ has points for every $\mathscr{D}$ which is Morita equivalent to $\mathscr{C}$ and an arrow for every Morita equivalence between $\mathscr{C}$ and $\mathscr{D}$ (up to equivalence of bimodule categories).  
\end{definition}

The Brauer-Picard groupoid is important in understanding the extension theory of $\mathscr{C}$ \cite{0909.3140}.

If $R$ is a ring and $M$ is an $R$-module, then putting an $R$-$S$ bimodule structure on $M$ is the same as giving a map of rings $S \rightarrow \mathrm{End}_R(M)$.  Furthermore, $M$ is invertible if and only if the corresponding map is an isomorphism.  Two such bimodules are equivalent if and only if the two maps $S \rightarrow \mathrm{End}_R(M)$ differ by conjugation by an invertible element in $S$.  Hence a Morita equivalence between $R$ and $S$ corresponds (non-canonically) to a pair consisting of an $R$-module $M$ whose commutant is $S$ and an outer automorphism of $S$.

The same story holds in the categorified setting as well.  If $\mathscr{M}$ is a left $\mathscr{C}$ module category, then giving $\mathscr{M}$ the structure of a $\mathscr{C}$--$\mathscr{D}$ bimodule category is the same thing as giving a monoidal functor $\mathscr{D} \rightarrow \mathscr{C}_\mathscr{M}^*$ (where $\mathscr{C}_\mathscr{M}^*$ is the category of $\mathscr{C}$ module category endofunctors of $\mathscr{M}$).  Furthermore, a $\mathscr{C}$--$\mathscr{D}$ bimodule category is invertible if and only if the monoidal functor $\mathscr{D} \rightarrow \mathscr{C}_\mathscr{M}^*$ is an equivalence of monoidal categories.    Finally, two $\mathscr{C}$--$\mathscr{D}$ bimodule categories are equivalent, if and only if the corresponding functors $\mathscr{D} \rightarrow \mathscr{C}_\mathscr{M}^*$ differ by an inner automorphism (that is by conjugation by an invertible object in $\mathscr{D}$).  Hence, Morita equivalences between $\mathscr{C}$ and $\mathscr{D}$ correspond (non-canonically) to a pair of a $\mathscr{C}$ module category $\mathscr{M}$ and an outer automorphism of $\mathscr{D}$.

The above result is important because it allows us to easily count the number of Morita equivalences between two categories.  In particular, we will prove that $\mathscr{H}_2$ has exactly three module categories and that their duals are $\mathscr{H}_1$, $\mathscr{H}_2$, and $\mathscr{H}_3$.  It follows that the number of Morita equivalences between $\mathscr{H}_2$ and $\mathscr{H}_i$ is just the number of outer automorphisms of $\mathscr{H}_i$.  Thus, the number of outer automorphisms is the same for each of the $\mathscr{H}_i$.  In particular, we will prove that there are no nontrivial outer automorphisms of $\mathscr{H}_2$ and thus that there are no nontrivial outer automorphisms of $\mathscr{H}_1$ and $\mathscr{H}_3$.  Hence, using the same counting argument again, there is exactly one module category over each of the $\mathscr{H}_i$ whose dual is each of the $\mathscr{H}_j$.

\subsection{Algebra objects}

In this subsection we sketch the definition of an algebra object in a fusion category and the relationship between algebra objects and module categories.  For more details see \cite[\S 3]{MR1976459}.

\begin{definition}
An algebra object in a fusion category  $\mathscr{C}$ is an object $A$ together with a unit morphism $\mathbf{1} \rightarrow A$ and a multiplication morphism $A \otimes A \rightarrow A$ satisfying the usual compatibility relations (associativity, left unit, right unit).
\end{definition}

A left module object over $A$ is defined in the obvious way, namely it is an object $M$ together with a morphism $A \otimes M \rightarrow M$ satisfying the usual compatibility relations.  Similarly we can define a right module object over an algebra $A$, and an $A$-$B$ bimodule object over two algebras.

A particularly important example of an algebra object is the internal endomorphisms $\underline{Hom}(M, M)$ with composition as the algebra structure.  The unit structure comes from the identity morphism from $M$ to $M$ via the definition of the internal hom.

Note that if $A$ is an algebra object in $\mathscr{C}$ then the category of left $A$-modules in $\mathscr{C}$ is a \emph{right} $\mathscr{C}$-module category.  Similarly the category of right $A$-modules in $\mathscr{C}$ is a left $\mathscr{C}$-module category.  A key theorem of Ostrik's gives a converse to this fact.

\begin{theorem} \cite[Theorem 1]{MR1976459}
Let $\mathscr{M}$ be a simple left (resp. right) module category over a fusion category $\mathscr{C}$ and let $X$ be any simple object in $\mathscr{M}$, then $\mathscr{M}$ is equivalent as a module category to the category of right (resp. left) $\underline{Hom}(X, X)$ modules in $\mathscr{C}$.
\end{theorem}

Note that a simple algebra $A$ is the internal endomorphisms of itself as an $A$-module, so the simple algebras in $\mathscr{C}$ are precisely the $\underline{Hom}(X,X)$ for $X$ in some simple module category over $\mathscr{C}$.

\begin{definition}
We call a simple algebra object $A$ in $\mathscr{C}$ \emph{minimal} if its Frobenius-Perron dimension is minimal among all $\underline{Hom}(X, X)$ where $X$ varies over simple objects in the category of right $A$-modules.
\end{definition}

\begin{corollary}
Every simple module category can be realized as the category of modules over a minimal algebra object.
\end{corollary}

If $X$ is an invertible object in a fusion category $\mathscr{C}$, then conjugation by $X$ gives an automorphism of the fusion category $\mathscr{C}$.  Thus if $A$ is an algebra object, then so is $X \otimes A \otimes X^*$ (here the multiplication is just given by contracting the middle $X^* \otimes X$ and then multiplying in $A$).

\begin{lemma} \label{lem:inneraut}
If $\mathscr{C}$ is a fusion category, $A$ is an algebra object, and $X$ is an invertible object, then the category of left (resp. right) $A$-modules and the category of left (resp. right) $X \otimes A \otimes X^*$-modules are isomorphic as right (resp. left) module categories.
\end{lemma}
\begin{proof}
The functor is given by $V \mapsto X \otimes V$ (resp. $V \mapsto V \otimes X^*$) where the action of $X\otimes A \otimes X^*$ on $X \otimes V$ is given by contracting the middle factor and then acting $A$ on $V$.  The binatural transformation is just given by the associator $(X\otimes V) \otimes W \rightarrow X\otimes (V \otimes W)$.  The inverse functor is given by $V \mapsto X^* \otimes V$.  
\end{proof}

Thus in order to classify all simple module categories it is enough to classify minimal algebra objects $A$ up to inner automorphism.  This result is useful because $\mathscr{H}_2$ has several invertible objects.

As pointed out to us by the referee, Lemma \ref{lem:inneraut} does not hold for arbitrary automorphisms of $\mathscr{C}$.  For example, if $\mathscr{C}$ is the category of $G$-graded vector spaces and $A$ is the group ring of $H \subset G$, then any subgroup conjugate to $H$ gives the same module category (namely $\mathrm{Vec}(G/H)$), while outer automorphisms (either coming from outer automorphisms of $G$ itself, or coming from a nontrivial $2$-cocycle) typically give a different module category.  In the special case where $G$ is abelian (so $\mathrm{Vec}(G) \cong \mathrm{Rep}(\hat{G})$) see  \cite[Theorem 2]{MR1976459}.

\subsection{Subfactors} \label{sec:subfactorvsalgebra}

In this subsection we recall the relationship between subfactors and algebra objects.  For more details see \cite{MR1257245, MR1444286, MR1966524, MR1976459}.

A subfactor is a unital inclusion $N \subset M$ of von Neumann algebras with trivial centers.  The index measures the dimension of $M$ as an $N$-module.  We will only consider subfactors of a Type II$_1$ factor with finite index.  We call $N \subset M$ irreducible if $M$ is irreducible as an $M$-$N$ bimodule.  Given a subfactor $N \subset M$ its \emph{principal even part} is the monoidal category of $N$-$N$ bimodules which appear as summands of tensor powers of ${}_NM_N$, the \emph{dual even part} is the monoidal category of $M$-$M$ bimodules which appear as summands of tensor powers of ${}_M M\otimes_N M_{M}$.   These two categories have the structure of $C^*$-tensor categories; the subfactor is said to have \emph{finite depth} if they are fusion categories, i.e. if there are only finitely many simple objects up to isomorphism. A $C^*$-fusion category is also called a unitary fusion category.

Furthermore, given a subfactor $N \subset M$ with even parts $\mathscr{C}$ and $\mathscr{D}$ we have a Morita equivalence between $\mathscr{C}$ and $\mathscr{D}$ given by the category of $N-M$ bimodules generated by ${}_N M_M $. 

Let $\kappa$ denote the $M-N$ bimodule ${}_M M _N$. We will often use sector notation: objects are labeled by Greek letters, square brackets denote isomorphism classes,  and tensor symbols are omitted, so that e.g. $\bar{\kappa} \kappa $ means ${}_N M _M \otimes_M {}_M M _N $. Then $\bar{\kappa} \kappa $ is an algebra object in $\mathscr{C}$, and the index of the subfactor is the dimension of the algebra object. Conversely, any rigid $C^*$ tensor category with countably many simple objects and simple identity object arises as a category of finite index $N$-$N$ bimodules over a factor \cite{MR1444286, MR1960417}. Moreover, any algebra object whose multiplication is a scalar multiple of a coisometry can be realized as $\bar{\kappa} \kappa $ for some subfactor $N \subseteq M$; such algebra objects are called Q-systems \cite{MR1257245, MR1444286}.

In general, it is not clear whether every simple algebra object in a $C^*$ fusion category gives a Q-system (as multiplication may not be a multiple of a coisometry).  However, for the Haagerup fusion categories that we consider every simple algebra object does indeed give a Q-system; thus we often elide the distinction in the statements of the major theorems.   Nonetheless we do need to explicitly check that certain algebra objects give Q-systems.  For these results the following lemma will be useful.

\begin{lemma} \label{lem:QTransfer}
Suppose that $\mathscr{C}$ is a $C^*$ fusion category and that $A$ is a Q-system in $\mathscr{C}$.  Let $\mathscr{M}$ be the left $\mathscr{C}$-module category of right $A$ modules, and let $X$ be a simple object in $\mathscr{C}$.  Then $\underline{\mathrm{End}}(X)$ is a $Q$-system in $\mathscr{C}$.
\end{lemma}
\begin{proof}
Since $A$ is a $Q$-system, there's a $2$-$C^*$-category (in the sense of \cite[\S 7]{MR1444286}) whose objects are $1$ and $A$, whose $1$-morphisms are the bimodule objects over those algebras, and whose $2$-morphisms are maps of bimodules.  In this context $\underline{\mathrm{End}}(X)$ becomes $\overline{X} X$ where we think of $X$ as a $1$-morphism between $1$ and $A$.  Hence, by \cite[\S 7]{MR1444286}, $\underline{\mathrm{End}}(X)$ is a $Q$-system.
\end{proof}

The right way to think of the above result is that it says that there's a good notion of $C^*$ module categories over $C^*$ fusion category $\mathscr{C}$, and that the $Q$-system condition just says that the corresponding module category is a $C^*$ module category.  Thus we need only check the $Q$-system condition once per module category.

\begin{remark}
Note that in the literature the generalization of $Q$-systems to the nonalgebraic context is typically that of a Frobenius Algebra \cite{MR1966524, MR2075605}.  However, in context of simple algebras, the Frobenius trace just comes from the (unique up to rescaling) splitting of the unit morphism.
\end{remark}

We recall the definition of the principal graphs of a subfactor. For any two objects $\rho$ and $\sigma$ in a fusion category or $C^*$ tensor category $\mathscr{C}$, let $(\rho, \sigma )= dim(Hom( \rho, \sigma ))$. 

Let $N \subset M$ be a finite index subfactor, with principal and dual even parts $\mathscr{N}$ and $\mathscr{M}$. Let $\kappa= {}_M M_N$, and let $\mathscr{K} $ be the category of $M-N$ bimodules generated by $\kappa \mathscr{N}$. The principal graph of the subfactor is the bipartite graph with even vertices indexed by the simple objects of $\mathscr{N}$, odd vertices indexed by the simple objects of $\mathscr{K}$, and for any pair of simple objects $\xi \in \mathscr{N}, \eta \in \mathscr{K}$, $(\kappa \xi,\eta ) $ edges connecting the corresponding vertices. It can be made into a pointed graph by distinguishing the even vertex corresponding to the identity object in $\mathscr{N}$, which is denoted by $*$. The dual graph is defined the same way but with $\mathscr{M}$ replacing $\mathscr{N}$ and $\bar{\mathscr{K}}= \mathscr{M} \kappa $. 

If the subfactor has finite depth, then the Frobenius-Perron dimensions of the objects are given by the Frobenius-Perron weights of the graphs, normalized to be $1$ at $*$, and the index of the subfactor is the squared norm of the graph.

Moreover, if $\kappa$ is any object in a semisimple module category over a fusion category, we can define the principal graph the same way.

\subsection{The Haagerup subfactor}

The Haagerup subfactor \cite{MR1686551} is a finite-depth subfactor with index $\frac{5+\sqrt{13}}{2} $; this is the smallest index above $4$ for any finite depth subfactor \cite{MR1317352}. The Haagerup subfactor is unique, up to duality. It has the following principal and dual graphs:

$$\hpic{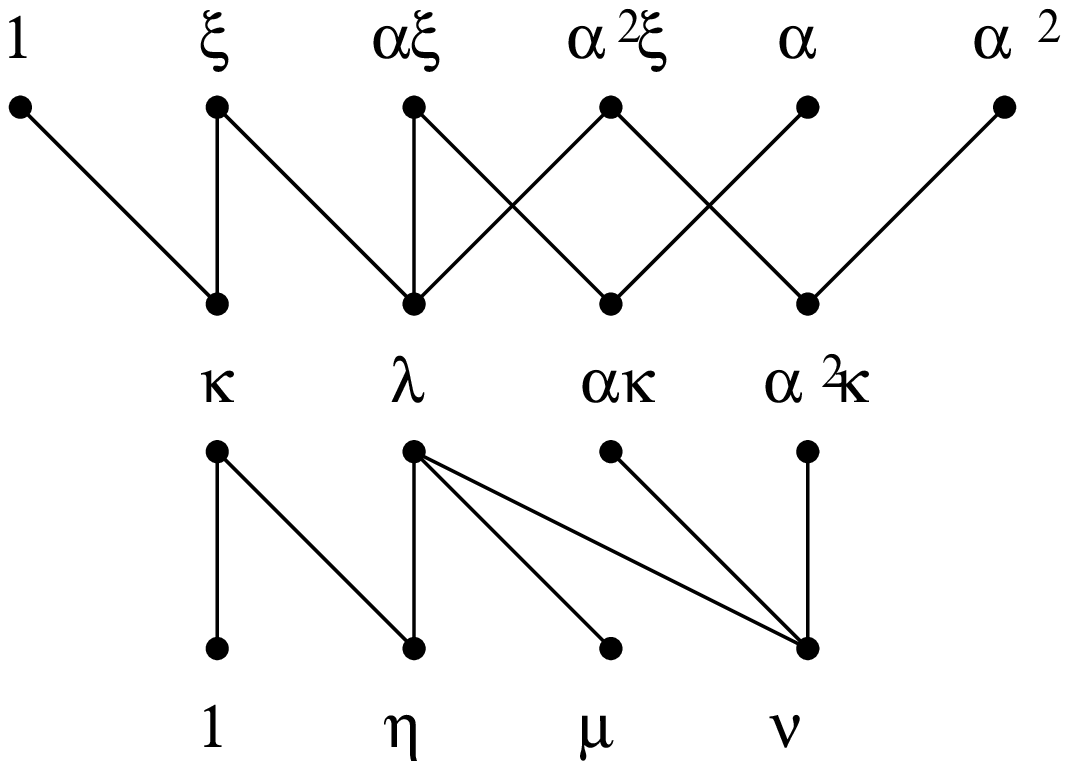} {2in} $$

We will call the fusion category with four simple objects $\mathscr{H}_1$ and the one with six simple objects $\mathscr{H}_2 $. The Frobenius-Perron dimensions of the simple objects are $d(\alpha)=1$, $d = d(\xi)=d(\eta)=d(\mu)+1=d(\nu)-1=\frac{3+\sqrt{13}}{2}$, $d(\kappa)=\sqrt{d+1} $, $d(\lambda)=\sqrt{(d+1)(d+2)} $. 

The fusion ring for $\mathscr{H}_2$ will be called $H_6$; it satisfies the relations $[\alpha^3]=[1], [\alpha \xi] = [\xi \alpha^2]$, and $[\xi^2]=[ 1 ]\oplus[\xi ] + [\alpha \xi ] + [ \alpha^2 \xi]$. The fusion ring for $\mathscr{H}_1 $ will be called $H_4$; it is commutative and the fusion rules are determined by the ring homomorphism property of dimension, self-duality of all simple objects, and Frobenius reciprocity.

\begin{table}
\begin{tabular}{ c || c | c | c | c | c | c }

                 & $1$             & $\alpha$       & $\alpha^2 $ & $\xi$          & $\alpha \xi$     & $\alpha^2 \xi $  \\ \hline \hline
$1$              & $1$             & $\alpha$       & $\alpha^2 $ & $\xi$          & $\alpha \xi$     & $\alpha^2 \xi $ \\ \hline
$\alpha$         & $\alpha$        & $\alpha^2$     & $1 $        & $\alpha \xi$   & $ \alpha^2 \xi$  & $\xi $ \\ \hline
$\alpha^2 $      & $\alpha^2$      & $1$            & $\alpha $   & $\alpha^2 \xi$ & $\xi$            & $\alpha \xi $ \\ \hline
$\xi$            & $\xi$           & $\alpha^2 \xi$ & $\alpha \xi$& $1+Z$ & $\alpha^2+Z$ & $\alpha+Z$ \\ \hline
$\alpha \xi $    & $\alpha \xi$    & $\xi$           & $\alpha^2 \xi $ & $\alpha+Z$ & $1+Z$ & $\alpha^2+Z$ \\ \hline
$\alpha^2 \xi  $ & $\alpha^2 \xi$  & $\alpha \xi$   & $\xi $ & $\alpha^2+Z$ & $\alpha+Z$ & $1+Z$ \\
\end{tabular}
\caption{ $H_6$ multiplication table.  We use the abbreviation $Z = \xi+\alpha\xi+\alpha^2\xi$}
\end{table}
\begin{table} 
\begin{tabular}{ c || c | c | c | c}

            & $1$   & $\nu$                  & $\eta $ & $\mu$    \\ \hline \hline
 $1$  	    & $1$   & $\nu$                  & $\eta $ & $\mu$    \\ \hline 
  $\nu$     & $\nu$ & $1+2\nu+2\eta+\mu$     & $2\nu+\eta+\mu $ & $\nu+\eta+\mu$    \\ \hline 
   $\eta $  & $\eta$ & $2\nu+\eta+\mu$                 & $1+\nu+\eta+\mu$ & $\nu+\eta$    \\ \hline
  $\mu$     & $\mu$  & $\nu+ \eta+\mu$       & $\nu+\eta $ & $1+\nu$    \\ 
\end{tabular}
\caption{ $H_4$ multiplication table}
\label{h4mult}
\end{table}

The fusion category $\mathscr{H}_2$ has two non-trivial invertible objects: $\alpha$ and $\alpha^2$.  The inner automorphism given by conjugation by $\alpha$ cyclically permutes $\xi$, $\alpha \xi$, and $\alpha^2 \xi$.  

The full subcategory generated by $\alpha$ has three invertible objects and is thus monoidally equivalent to $\mathrm{Vec}(\mathbb{Z}/3\mathbb{Z}, \omega)$ for some associator $\omega \in H^3(\mathbb{Z}/3, \mathbb{C}^*)$.  In fact, as was pointed out to us by David Jordan, this cocycle must be trivial.

\begin{lemma} \label{lem:cocycle}
The full subcategory generated by $\alpha$ is equivalent as a fusion category to the category of $\mathbb{Z}/3\mathbb{Z}$-graded vector spaces.  
\end{lemma}
\begin{proof}
Notice that since $\alpha \lambda = \lambda$, the category of vector spaces (thought of as sums of $\lambda$) is a module category over $\mathscr{D}$.  Hence $\mathscr{D}$ has trivial associator.
\end{proof}

\section{Algebra objects, principal graphs and subfactors in the Haagerup categories}

The goal of this section is to classify all simple algebra objects in each of the $\mathscr{H}_i$, and to classify all indecomposable module categories over each of the $\mathscr{H}_i$.  The outline of the argument is as follows.  We use combinatorics to describe the possible objects which could have a simple algebra structure.  However, this list is somewhat large, and since some of the objects are relatively complex it is difficult to determine how many algebra structures each such object admits. Fortunately, in order to classify all module categories it is enough to only consider the algebra objects whose dimensions are minimal among all algebras which yield the same module category.  Furthermore we need only consider these algebra objects up to inner automorphisms of the category.

There are many fewer candidates for these minimal algebra objects and we are able to easily determine when the algebra structure exists and that it is unique when it does exist.  Thus we obtain a complete list of all indecomposable module categories, and using the internal Hom construction we are able to read off the full list of (not necessarily minimal) simple algebra objects.   We do this first for $\mathscr{H}_2$, and then use this classification to read off the same classification for the other $\mathscr{H}_i$.

In essence what we are doing is moving back-and-forth between algebra objects and module categories in order to exploit the more accessible combinatorial structure of algebra objects and the more rich algebraic structure of module categories.  This interplay allows us to avoid computations which would otherwise be extremely difficult.  To illustrate the general technique we prove the following result which was proved with considerable effort in the appendix of \cite{MR2418197}.

\begin{lemma}
There exists a unique simple algebra structure on $1 + \nu$ in $\mathscr{H}_1$.  This algebra is also a Q-system.
\end{lemma}
\begin{proof}
Since $1 + \nu \cong \mu \bar{\mu}$, it has at least one algebra structure given by contraction.  Furthermore, since the algebra $1$ is a $Q$-system and $\mu$ is an object in the category of $1$-$1$ bimodules, by Lemma \ref{lem:QTransfer} this algebra is a $Q$-system.

Now suppose that we have any algebra structure on $1 + \nu$.  A simple combinatorial calculation (see Example \ref{ex}) shows that the principal graph of the corresponding subfactor must be \hpic{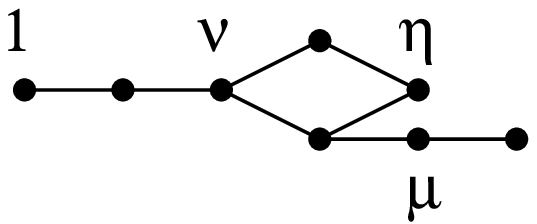} {.3in}

The vertex all the way on the right is an odd vertex of dimension $1$ which we call $ \theta$.  Note that as an odd vertex, $\theta$ is a simple object in the category of $(1 + \nu)$-modules.  A dimension count shows that $\underline{\mathrm{Hom}}(\theta,\theta) \cong 1$, and so the category of $(1+ \nu)$-modules is equivalent to the category of $1$-modules which is just  $\mathscr{H}_1$ itself with the usual module action.  Hence, the algebra structure on $1 + \nu$ can be realized as the internal endomorphisms of some object in $\mathscr{H}_1$.  A dimension count shows that this object must be $\mu$, hence we have that $1 + \nu \cong \underline{\mathrm{Hom}}(\mu,\mu) \cong \mu \bar{\mu}$ as algebra objects.
\end{proof}

We now turn to the general question of classifying all simple algebra objects in and all simple module categories over $\mathscr{H}_1$ and $\mathscr{H}_2$.

Let $\mathscr{C}$ be a fusion category, and let $ \mathscr{K}_{\mathscr{C}}$ be a module category over $\mathscr{C}$. Let $\xi_0=1,\xi_1,...\xi_n$ be an enumeration of the simple objects in $\mathscr{C}$, and let $\kappa_0, \kappa_1, ... \kappa_m$ be an enumeration of the simple objects in $\mathscr{K}$. Let $\kappa$ be an object in $\mathscr{K}$. 

\begin{definition}
 The fusion matrix of $\kappa$ is the matrix $(F^{\kappa}_{ij} )_{0 \leq i \leq n, 0 \leq j \leq m}$ given by $F^{\kappa}_{ij}=( \kappa \xi_i, \eta_j) $.
\end{definition}

Note that the fusion matrix is only defined up to a choice of ordering of the simple objects. 

\begin{example}
 Let $N \subset M$ an irreducible, finite-index, finite depth subfactor, and let $\kappa= {}_M M_N$. Then the fusion matrix $A=F^{\kappa}$ is an adjacency matrix of the principal graph, with the first row corresponding to $*$. In this case we take the convention $\kappa_0=\kappa$, so that the first column of the matrix gives the edges emanating from $\kappa$.
\end{example}

We would like to figure out which objects in a fusion category can admit a simple algebra structure. By an abuse of notation, we will often omit reference to the algebra structure and refer to an object $\gamma$, or even its isomorphism class $[\gamma]$, as an algebra.

\begin{lemma}
Let $\gamma= 1 + \sum_{i=1}^n \limits a_i \xi_i  $ be a simple algebra object in a fusion category $\mathscr{\mathscr{C}}$. Let $\kappa$ be a simple right $\gamma$-module such that $\gamma \cong \underline{\mathrm{Hom}}(\kappa,\kappa)$. Then $F^{\gamma}=AA^T $, where $A=F^{\kappa}$ is the adjacency matrix of the principal graph of $\kappa $ (with the same ordering of the simple objects of $\mathscr{C}$).
\end{lemma}
\begin{proof}
Fix $\kappa$ such that $\gamma \cong \bar{\kappa} \kappa $. Then we have $F^{\gamma}_{ij}=(\gamma \xi_i, \xi_j ) =( \bar{\kappa} \kappa \xi_i ,\xi_j )= (\kappa \xi_i, \kappa \xi_j)$. On the other hand, $(AA^T)_{ij}= \sum_{l}(\kappa \xi_i, \kappa_l )(\kappa \xi_j, \kappa_l )=(\kappa \xi_i, \kappa \xi_j )$. 
\end{proof}

We will call the graph given by $A$ a principal graph of the algebra object. Note that it is not uniquely determined by the object $\gamma$, although it is uniquely determined by the algebra structure. Nevertheless in many cases there is at most one possible choice for the graph given an object $\gamma$. If $\gamma = 1 + \sum_{i=1}^n \limits a_i \xi_i  $ is a simple algebra object, then as we have seen the first column of $A$ is given by the cofficients $1, (a_i)$ of $\gamma$, and the rest of the first row is $0$. It is therefore sometimes convenient to lop off the first row and column when doing computations.
\begin{definition}
 The reduced fusion matrix of $\gamma=1 + \sum_{i=1}^n \limits a_i \xi_i $,\\ $F^{\gamma,r}$,  is given by $F^{\gamma, r}_{ij}=F^{\gamma}_{ij} - a_i a_j=\sum_{k=0}^n \limits a_k (\xi_k \xi_i, \xi_j) -a_i a_j $ for $1 \leq i,j \leq n$. A reduced principal graph of $\gamma $ is a graph given by the matrix $A^r$, defined by $ A^r_{ij}=A_{ij}$ for $i,j \geq 1$, where $A$ is the adjacency matrix of a principal graph of $\gamma$. 
\end{definition}
\begin{lemma} \label{findgraphs}
We have $F^{\gamma, r}=A^r(A^r)^T $ (with the same ordering of simple objects of $\mathscr{C}$).
\end{lemma}
\begin{proof}
 By definition, we have $F^{\gamma, r}_{ij}=F^{\gamma}_{ij} - a_i a_j=(AA^T)_{ij}-a_ia_j=(A^r(A^r)^T)_{ij}$.
\end{proof}
This allows us to quickly find possible algebra objects in a fusion category by checking which objects have reduced fusion matrices that decompose as $AA^T$ for some matrix $A$ all of whose entries are nonnegative integers. Then the reduced principal graph is given by such an $A$, and the full principal graph is obtained by adding the vertices corresponding to $1$ and $\kappa$. 

We now apply this analysis to the Haagerup fusion categories. 

\begin{example} \label{ex}

We explain the method for a few objects in $\mathscr{H}_1$. We always use the following ordering of simple objects: $1, \nu,  \eta, \mu $.

(a) Let $\gamma=1 + \nu $. Then the reduced fusion matrix can be computed using \ref{h4mult}; it is: 
$\begin{pmatrix}
 2 & 2 & 1\\
2 & 2 & 1\\
 1 & 1 & 2\\ 
\end{pmatrix}$.
It is easy to see that up to graph equivalence, the only candidate for $A$ is: 
$\begin{pmatrix}
 1 & 1 & 0\\
1 & 1 & 0\\
 1 & 0 & 1\\
\end{pmatrix}$.
Adding a column with the coefficients of the simple objects in $\gamma$,  $1,1,0,0$, along with a row of zeros on top gives
$\begin{pmatrix}
1 & 0 & 0 & 0\\
 1 & 1 & 1 & 0\\
0 & 1 & 1 & 0\\
0 & 1 & 0 & 1\\ 
\end{pmatrix}$. Therefore the only graph compatible with an algebra structure on $\gamma$ is \hpic{h7} {0.4in} .
 
(b) $\gamma=1 + 2\nu$. Then the reduced fusion matrix is 
$\begin{pmatrix}
 1 & 4 & 2\\
4 & 3 & 2\\
 2 & 2 & 3\\ 
\end{pmatrix}$.
Since this matrix is not positive semi-definite, it does not decompose as $AA^T$ and $\gamma $ does not admit an algebra structure.

(c) Let $\gamma=1 + 4\nu + 3\eta + 2\mu$. Then the reduced fusion matrix is 
$\begin{pmatrix}
 1 & 1 & 1\\
1 & 1 & 1\\
 1 & 1 & 1\\ 
\end{pmatrix}$,
which decomposes as $AA^T$ only when $A$ is the matrix:
 $\begin{pmatrix}
 1\\
1\\
 1\\ 
\end{pmatrix}$;
the principal graph is then given by
$\begin{pmatrix}
1 & 0 \\
 4 & 1 \\
3 & 1 \\
2 & 1 \\ 
\end{pmatrix}$.
The Frobenius-Perron weight corresponding to the second column is $\sqrt{3}$. Suppose such an algebra object exists, and let $\kappa'$  be the simple object with dimension $\sqrt{3} $. Then $[\kappa' \bar{\kappa'}]=[1]+[ \sigma] $, where $[\sigma] $ is a nonnegative integral linear combination of $[\nu], [\eta], [\mu] $. Since $d(\sigma )=3-1=2$, this is impossible. Therefore $\gamma $ does not admit an algebra structure.
\end{example}

In order to turn the classification of simple algebra objects in a fusion category into a finite problem we need a bound on the size of possible simple algebra objects.

\begin{lemma}\label{lem:algebra-bound}
If $\gamma$ is a simple algebra object and $\xi$ is any simple object, then $(\gamma, \xi ) \leq \dim(\xi)$.
\end{lemma}
\begin{proof}
This result is well-known in the subfactor context, but we quickly prove it in order to see that it works in the fusion category context.  Recall that $\gamma = \eta \bar{\eta}$ for some simple object $\eta$ in a module category (namely $\eta$ is just $\gamma$ as a $\gamma$-module).  By Frobenius reciprocity $(\gamma, \xi) = (\eta, \xi \eta)$.  Since $\eta$ is simple, $(\eta, \xi \eta)$ just measures the number of copies in $\xi \eta$.  Since Frobenius-Perron dimensions are always positive, $(\eta, \xi \eta)$ is at most $\dim \xi \eta/\dim \eta = \dim \xi$.
\end{proof}

\begin{lemma}
Let $\gamma $ be a nontrivial simple algebra object in a fusion category with fusion ring isomorphic to $H_4$. Then any principal graph of $\gamma$ is one of the following seven graphs, with the indicated indices and indicated even objects:\\ 

(1) \hpic{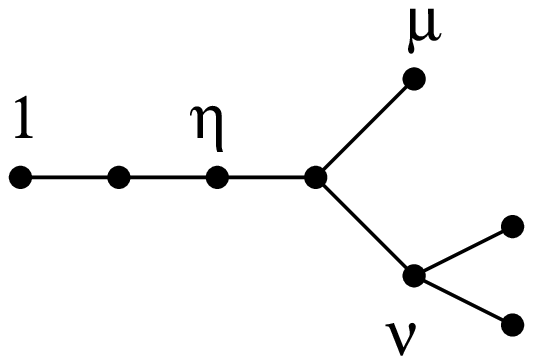} {0.6in}  $\displaystyle \frac{5+\sqrt{13}} {2}$, (2) \hpic{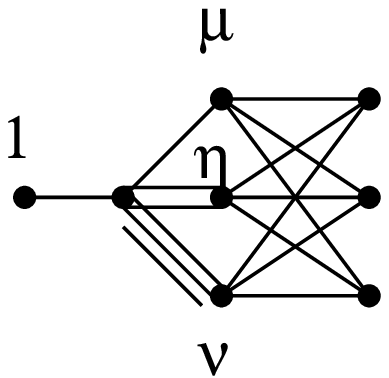} {0.6in} $12+3\sqrt{13}$ \\

(3) \hpic{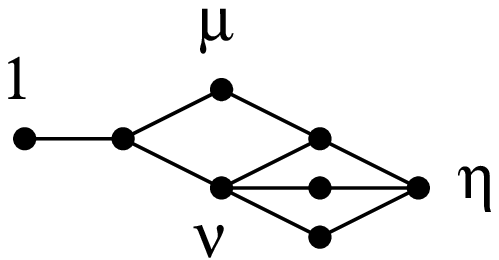} {0.6in}  $4+\sqrt{13}$, (4) \hpic{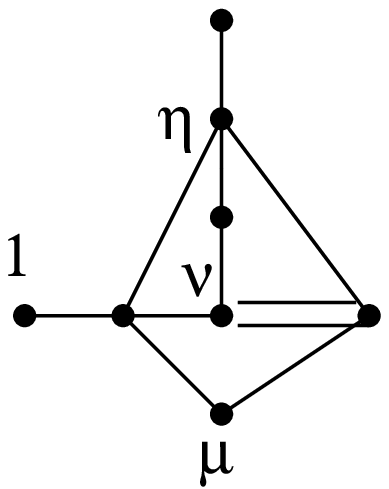} {1.1in}  $\displaystyle \frac{11+3\sqrt{13}} {2}$ \\

(5) \hpic{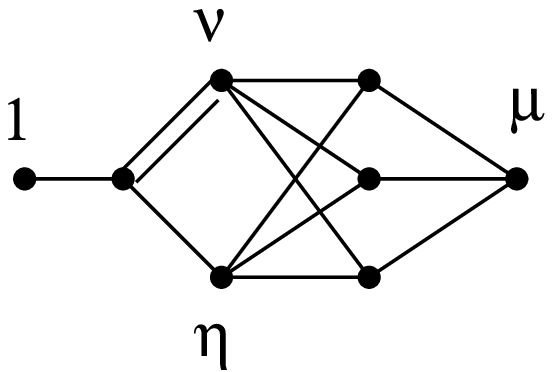} {0.6in} $\displaystyle \frac{15+3\sqrt{13}} {2}$, (6) \hpic{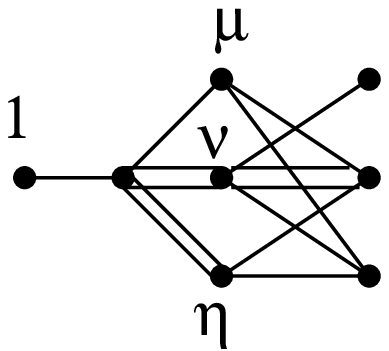} {0.6in} $\displaystyle \frac{19+5\sqrt{13}} {2}$ \\

(7) \hpic{h7} {0.4in} $\displaystyle \frac{7+\sqrt{13}} {2}$ \\

Furthermore, of the above algebra objects the only minimal ones are (1) and (3).
\end{lemma}

\begin{lemma} \label{lem:H6}
Let $\gamma$ be a nontrivial simple algebra object in a fusion category with fusion ring isomorphic to $H_6$. Then up to inner automorphism of the category any principal graph of $\gamma$ is one of the following seven graphs, with the indicated indices and even objects:\\

(1') \hpic{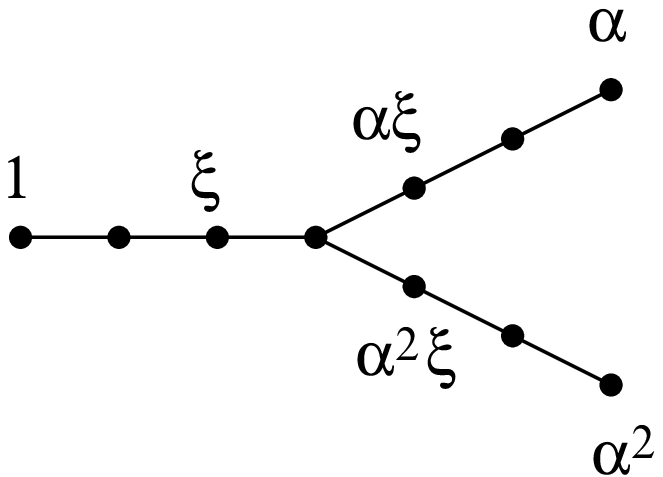} {0.6in}  $\displaystyle \frac{5+\sqrt{13}} {2}$, (2') \hpic{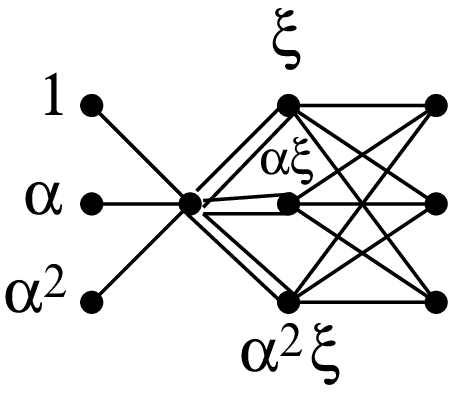} {0.6in}  $12+3\sqrt{13}$ \\

(3') \hpic{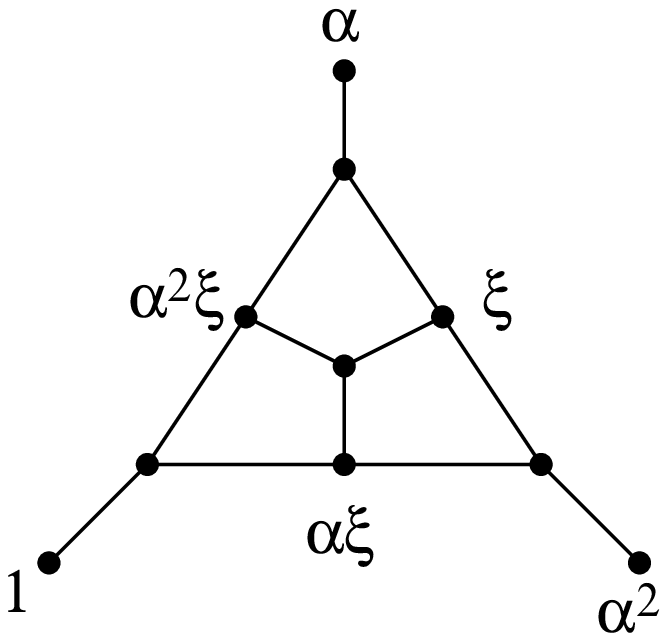} {1in}  $4+\sqrt{13}$, (4') \hpic{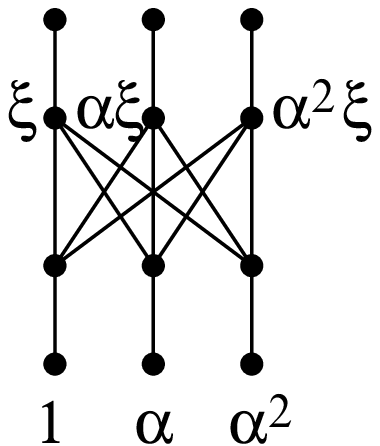} {1.1in}  $\displaystyle \frac{11+3\sqrt{13}} {2}$ \\

(5') \hpic{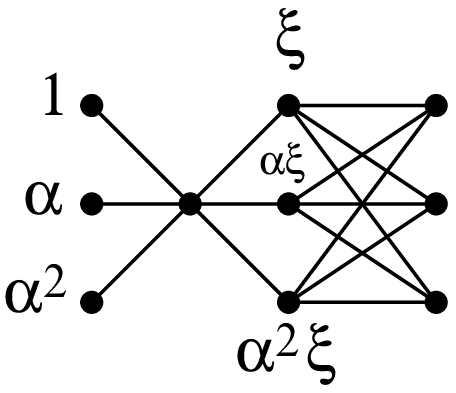} {0.5in}  $\displaystyle \frac{15+3\sqrt{13}} {2}$, (6') \hpic{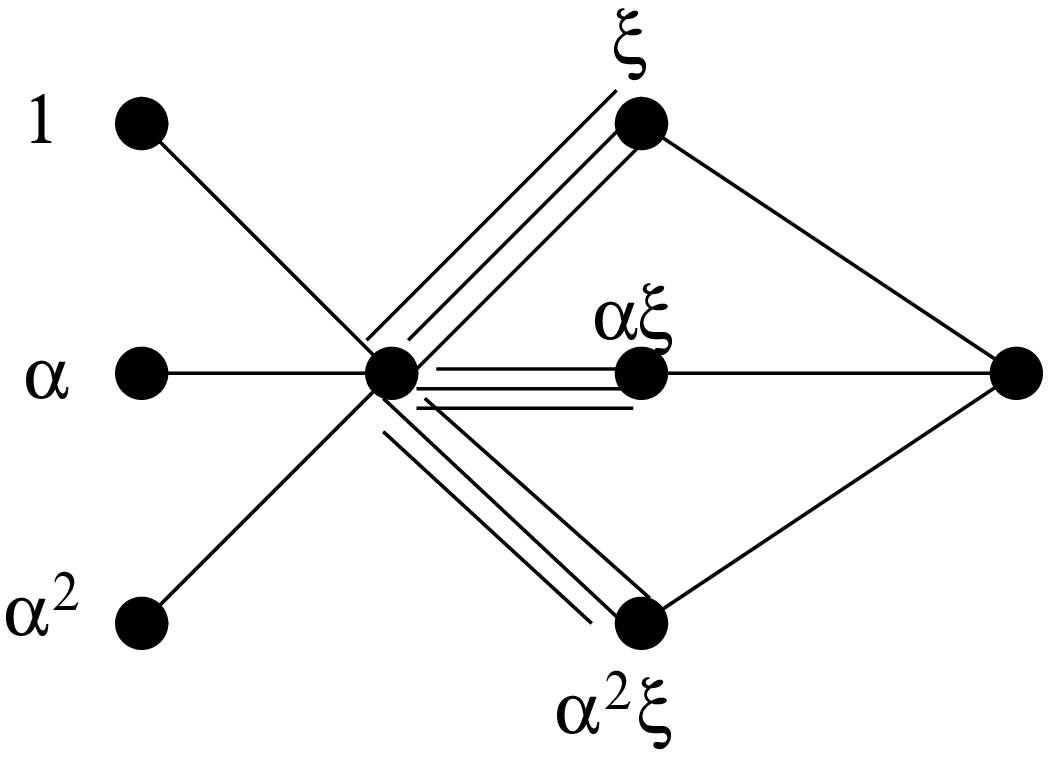} {0.6in}  $\displaystyle \frac{33+9\sqrt{13}} {2}$ \\

(7') \hpic{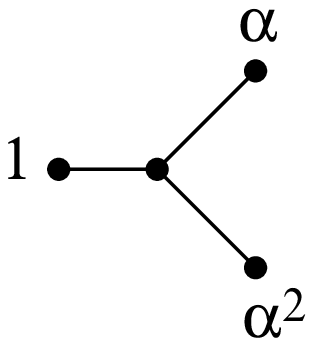} {0.4in}  $3$  \\

Furthermore, of the above algebra objects the only minimal ones are (1'), (3'), and (7').
\end{lemma}

\begin{proof}
Using Lemma \ref{lem:algebra-bound} and Lemma \ref{findgraphs} this is a tedious calculation. The only other admissible graphs encountered are two graphs each for $[1]+[\mu]+[\nu] $ and $[1] + 2[\mu] + [\nu] + [\eta] $, which are eliminated due to having odd vertices violating the Jones index restriction; and the graph mentioned in Example \ref{ex}, (c), which was eliminated there.
\end{proof}

\begin{corollary}
If $\mathscr{M}$ is a left module category over $\mathscr{H}_2$ then $\mathscr{M}$ is the category of right $A$ modules for some algebra structure on one of the objects $1$, $1+ \xi$, $1+\alpha+\alpha^2$, or $1+\xi + \alpha \xi$.
\end{corollary}
\begin{proof}
From the above lemma any minimal algebra object is of one of these forms. 
\end{proof}

\begin{lemma}
There exists a unique algebra object structure on the object $1$ in $\mathscr{H}_2$.  This algebra object is a $Q$-system.
\end{lemma}
\begin{proof}
The proof is immediate.
\end{proof}

\begin{lemma}
There exists a unique algebra object structure on the object $1+\xi$ in $\mathscr{H}_2$.  This algebra object is a $Q$-system.
\end{lemma}
\begin{proof}
Existence of a $Q$-system follows from the existence of the Haagerup subfactor.  Uniqueness of the algebra follows from $3$-supertransitivity.  Namely, since $\xi \xi$ only contains one copy of $1$ and only one copy of $\xi$, the multiplication and unit morphisms must lie inside Temperley-Lieb.  The uniqueness of the algebra structure inside Temperley-Lieb is a straightforward well-known calculation.
\end{proof}

\begin{lemma}
There exists a unique algebra object structure on the object $1+\alpha+\alpha^2$ in $\mathscr{H}_2$.  This algebra object is a $Q$-system.
\end{lemma}
\begin{proof}
$\mathscr{H}_2$ has a fusion full subcategory $\mathscr{D}$ consisting of sums of $1$, $\alpha$, and $\alpha^2$.  By Lemma \ref{lem:cocycle} this category is equivalent to $\mathbb{Z}/3\mathbb{Z}$-graded vector spaces.    Now existence and uniqueness follow from the same fact about the category of $\mathbb{Z}/3\mathbb{Z}$-graded vector spaces (which is essentially just the existence and uniqueness of the $D_4$ subfactor).
\end{proof}

\begin{lemma}\label{noalg}
There is no algebra object structure on  $1+ \xi + \alpha \xi$.
\end{lemma}
\begin{proof}
Suppose $1+\xi+\alpha \xi$ had an algebra object structure. Let $\kappa$ be $1+\xi+\alpha \xi$ as a left module over itself (so that $1+\xi+\alpha\xi \cong \kappa \bar{\kappa} $).  Let $\iota$ be $1+\alpha^2 \xi$ as a right module over itself (so that $1+\alpha^2 \xi \cong \bar{\iota } \iota$). Then $\iota \kappa $ is a $(1+\xi+\alpha\xi)-(1+\alpha^2 \xi)  $ bimodule. Moreover, by Frobenius reciprocity, $(\iota \kappa, \iota \kappa )=(\bar{\iota} \iota, \kappa \bar{\kappa} ) $, so $\iota \kappa $ is irreducible. Then  $\iota \kappa \bar{\kappa} \bar{\iota}$ must be a simple algebra object in the category of $(1+\alpha^2 \xi)-(1+\alpha^2 \xi)$-bimodules, which is $\mathscr{H}_1$.  The algebra $\iota \kappa \bar{\kappa} \bar{\iota}$ has dimension $\mathrm{dim}(1+\alpha^2 \xi )\cdot \mathrm{dim}(1+\xi+\alpha\xi )=\frac{33+9\sqrt{13}}{2} $. But there is no admissible principal graph in $\mathscr{H}_1 $ with that index. 

\end{proof}

\begin{corollary}
There are exactly three simple module categories over $\mathscr{H}_2$, namely the category of $1$-modules, the category of $(1+\xi)$-modules and the category of $(1+\alpha+\alpha^2)$-modules for each of the above algebra structures.  These module categories have the following graphs for fusion with any of the objects of dimension $\frac{3+\sqrt{13}}{2}$.

\hpic{m4} {.6in} \quad
\hpic{m5} {.6in} \quad
\hpic{m6} {.4in}
\end{corollary}
\begin{proof}
We have classified all minimal algebra objects up to inner automorphism of the fusion category; thus we have classified all simple module categories.  The fusion graphs can be read off directly from the principal graphs.
\end{proof}

From the above it is also easy to read off the complete list of simple algebra objects in $\mathscr{H}_2$, and thus all of the subfactors whose principal even part is $\mathscr{H}_2$.  Furthermore, their principal graphs can be easily identified on the list in Lemma \ref{lem:H6}. However, it is a bit more work to identify the dual even parts and the dual principal graphs.

\begin{lemma} \label{lem:dualspart1}
The category of $1$-$1$ bimodules in $\mathscr{H}_2$ is $\mathscr{H}_2$.
\end{lemma}
\begin{proof}
Obvious.
\end{proof}

\begin{lemma}\label{lem:dualspart2}
The category of $(1+\xi)$-$(1+\xi)$ bimodules in $\mathscr{H}_2$ is $\mathscr{H}_1$.
\end{lemma}
\begin{proof}
This follows from the definitions of $\mathscr{H}_2$ and $\mathscr{H}_1$ as the even parts of the Haagerup subfactor.
\end{proof}

\begin{definition} \label{lem:dualspart3}
Let $\mathscr{H}_3$ be the category of $(1+\alpha+\alpha^2)$-$(1+\alpha+\alpha^2)$ bimodules in $\mathscr{H}_2$.
\end{definition}

Note that we do not yet know that $\mathscr{H}_3$ is distinct from $\mathscr{H}_1$ and $\mathscr{H}_2$.

\begin{definition}
 If $\mathscr{C}$ and $\mathscr{D}$ are fusion categories, a subfactor $N \subset M$ will be called a $\mathscr{C}-\mathscr{D}$ subfactor if its principal even part is  $\mathscr{C}$ and its dual even part is $\mathscr{D}$.  Wild cards will be used when one of the categories is unknown or unspecified.
\end{definition}

\begin{lemma}\label{existence}
The object $1+\mu+\nu$ in $\mathscr{H}_1$ has a $Q$-system structure.  The corresponding subfactor is an $\mathscr{H}_1$-$\mathscr{H}_3$ subfactor and its principal graphs are (3) and (3').
\end{lemma}
\begin{proof}
 Let $N \subset M$ be an $\mathscr{H}_1-\mathscr{H}_2 $ subfactor with principal graphs (2) and (2'). From the graph (2') we see that there is an intermediate subfactor $N \subset P \subset M$ with $[M:P] = 3$.  Note that, by definition of $\mathscr{H}_3$, we have that $P \subset M$ induces a Morita equivalence between $\mathscr{H}_3$ and $\mathscr{H}_2$.  Thus, $N \subset P $ is an $\mathscr{H}_1-\mathscr{H}_3 $ subfactor which must have principal graph (3), since that is the only $H_4$ compatible graph with the right index. 

\begin{figure} 
\begin{center}
\hpic{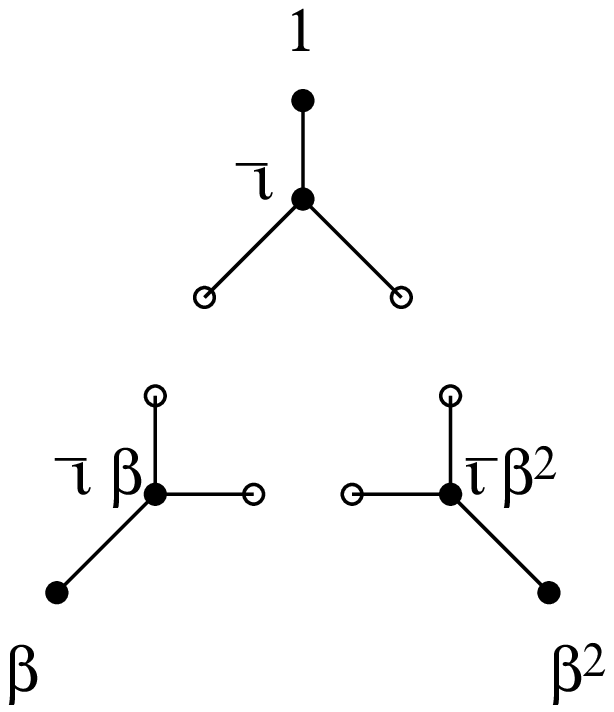} {1.2in}  $\rightarrow$ \hpic{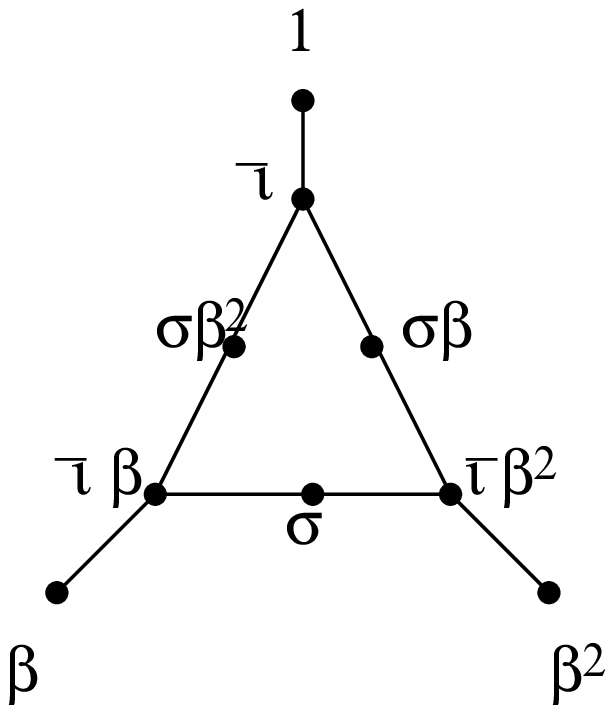} {1.2in} $\rightarrow$ \hpic{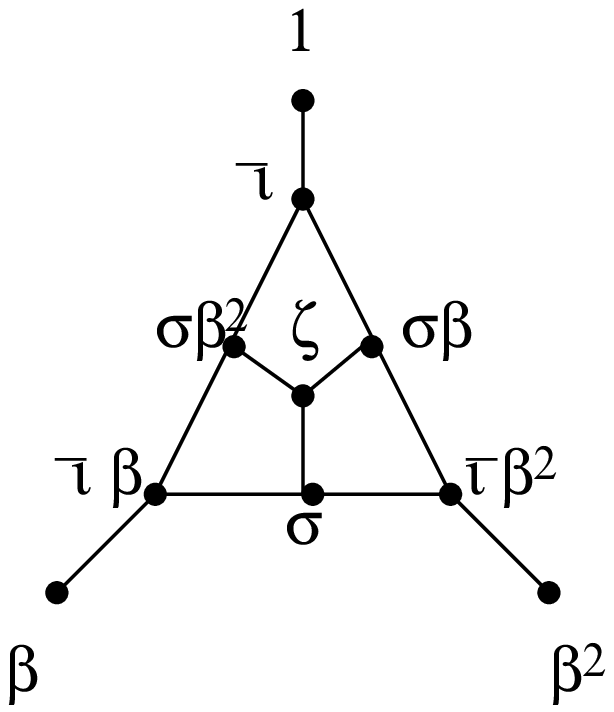} {1.2in}
\caption{Schematic representation of the graph computation} \label{buildinggraph}
\end{center}
\end{figure}

To compute the dual graph, let $\iota={}_N P_P $ and $\kappa={}_P M_M $. Note that $P \subset M $ is an $*-\mathscr{H}_2$ subfactor, so it must have dual graph (7'), which implies that is also has principal graph (7'). Let $\beta$ be a dimension one $P-P$ bimodule such that $[\kappa \bar{\kappa}]=[1] + [\beta] + [\beta^2]$. Then by Frobenius reciprocity, $(\bar{\iota} {\iota}, \kappa \bar{\kappa} )=(\iota \kappa, \iota \kappa ) =1$. In a similar way we find $([(\bar{\iota} {\iota})^2], [\beta] + [\beta^2])=2 $. This implies that there are two other vertices in the dual graph (labeled by $\beta$ and $\beta^2$ sharing an order three symmetry with $*$ (the vertex labeled by $1$), such that there are no length two paths but there are two length four paths from $*$ to these vertices. Moreover, by Frobenius reciprocity, the vertex corresponding to the fundamental object in the dual graph is trivalent, and therefore so are the (unique) vertices adjacent to the other two dimension one vertices. (See Figure \ref{buildinggraph}.) 
Taken together, this implies that the graph contains a triangle of odd vertices, with one even vertex in the middle of each side (which we label by $\sigma, \beta \sigma, \beta^2 \sigma$) and one even vertex (labeled by $1,\beta,\beta^2$) attached by an edge to each corner of the triangle (labeled by $\bar{\iota}, \beta \bar{\iota}, \beta^2 \bar{\iota}$). There is one additional odd vertex (which we label by $\zeta$), which must be connected to the vertices labeled by $\sigma, \beta \sigma, \beta^2 \sigma $. By symmetry, it must be connected to all three of them by at least one edge. Adding these three edges to the graph gives the correct index, so the graph can't be any larger. 
\end{proof}

\begin{corollary}
 $\mathscr{H}_3 $ is not equivalent to $\mathscr{H}_1 $ or $\mathscr{H}_2$.
\end{corollary}

\begin{proof}
 Clearly, $\mathscr{H}_3 $ is not equivalent to $\mathscr{H}_1 $ (they are already different at the level of fusion rings). Suppose $\mathscr{H}_3 $ were equivalent to $\mathscr{H}_2$. Then from the dual graph (3'), we see that $1+\xi+\alpha \xi $ would have to have an algebra structure. But by Lemma \ref{noalg}, that is not the case.
\end{proof}

\begin{lemma}\label{fring}
 Let $N \subset M $ be a subfactor with principal graphs (3) and (3'). Then its dual even part has fusion ring isomorphic to $H_6$.
\end{lemma}
\begin{proof}
 From the graph (3'), we see that the $M-M$ fusion ring contains three dimension one objects (as before we will label them by $1, \beta, \beta^2$), which generate a subring isomorphic to $\mathbb{Z} / 3\mathbb{Z}$. Then there are there are three objects of dimension $\frac{3+\sqrt{13}}{2} $ (as before labeled $\sigma, \beta \sigma, \beta^2 \sigma $), with $\beta \sigma$ and  $\beta^2 \sigma$ the two new objects at depth two. Then either $[\beta \sigma]=[ \bar{ \beta \sigma}] $ or $[\beta \sigma] =[ \bar{\beta^2 \sigma}]$. Since the two new objects in depth two in graph (3) have different dimensions, they are each self-conjugate, so the former holds by \cite[\S 3.3]{1007.1730}  (using an annular tangles argument \cite{MR1929335}). Then $[\beta \sigma] = [\bar{\beta \sigma}]=[ \sigma \beta^2 ]$, and the isomorphism to $H_6$ is easy to deduce.   
\end{proof}

\begin{corollary}
 There exists a unitary fusion category with fusion ring $H_6$ which is Morita equivalent to $\mathscr{H}_2 $ but not isomorphic to it.
\end{corollary}

%
%

\begin{theorem} \label{thm:H1H3}
There are exactly three simple module categories over each of $\mathscr{H}_1$ and $\mathscr{H}_3$.  In each case the three dual fusion categories are the three $\mathscr{H}_i$.  The simple module categories over $\mathscr{H}_1$ have the following graphs for fusion with the object of dimension $\frac{1+\sqrt{13}}{2}$.

\hpic{m1} {.3in} \quad \hpic{m2} {.6in} \quad \hpic{m3} {.6in}

The fusion graphs for the module categories over $\mathscr{H}_3$ (with respect to any of the $\frac{3+\sqrt{13}}{2}$ dimensional objects) are: 

\hpic{m4} {.6in} \quad
\hpic{m7} {.6in} \quad
\hpic{m6} {.4in}
\end{theorem}
\begin{proof}
Let $M$ be the number of Morita equivalences between $\mathscr{H}_i$ and $\mathscr{H}_j$ (this is clearly independent of $i$ and $j$), let $O_i$ be the number of outer automorphisms of $\mathscr{H}_i$, and let $B_{ij}$ be the number of simple module categories over $\mathscr{H}_i$ whose dual is isomorphic to $\mathscr{H}_j$.  We know that $M = O_j B_{ij}$ and that $B_{2j} = 1$ for any $i$ and $j$.  Hence, we have that $M = O_j B_{2j} = O_j$ for all $j$.  Thus we conclude that $B_{ij} = M/O_j = 1$.  Hence there are exactly three simple module categories over each $\mathscr{H}_i$.

The $\mathscr{H}_1$ module category whose dual is  $\mathscr{H}_1$ is trivial, and so its fusion graph can be read off from the fusion rules for $\mathscr{H}_1$.  The $\mathscr{H}_1$ module category whose dual is $\mathscr{H}_2$ comes from the Haagerup subfactor, and its fusion graph can be read off from fusion rules for the Haagerup subfactor.  Finally, the $\mathscr{H}_1$ module category whose dual is $\mathscr{H}_3$ can be read off from the dual principal graph (3').

The calculations for $\mathscr{H}_3$ are similar.
\end{proof}
\begin{corollary}
Any fusion category Morita equivalent to the principal even parts of the Haagerup subfactor must be isomorphic to $\mathscr{H}_1, \mathscr{H}_2 $ or $\mathscr{H}_3 $.
\end{corollary}

\begin{theorem} \label{subthm}
 The complete list of irreducible subfactors whose even parts are Morita equivalent to $\mathscr{H}_i$ is as follows; each exists and is unique up to isomorphism of the planar algebra. Note that the $\mathscr{H}_i-\mathscr{H}_i $ subfactors are self-dual while the $\mathscr{H}_i-\mathscr{H}_j $ subfactors for $i \neq j $ come in non-self-dual pairs; we only list each pair once.

(a) $\mathscr{H}_1-\mathscr{H}_2$: One with principal graphs (1)-(1') and one with principal graphs (2)-(2').

(b) $\mathscr{H}_1-\mathscr{H}_3$: One subfactor with principal graphs (3)-(3') and one with principal graphs (5)-(5').

(c) $\mathscr{H}_2-\mathscr{H}_3$: One subfactor with principal graphs (6')-(6'). 

(d) $\mathscr{H}_1-\mathscr{H}_1$: One subfactor each with principal graphs (4)-(4), (6)-(6), and (7)-(7).

(e) $\mathscr{H}_2-\mathscr{H}_2$: One subfactor with principal graph (4')-(4').

(f) $\mathscr{H}_3-\mathscr{H}_3$: One subfactor with principal graph (4')-(4').
\end{theorem}

\begin{proof}
 For each choice of $\mathscr{H}_i-\mathscr{H}_j$, we have a unique bimodule category up to isomorphism, so all the irreducible subfactors come from taking internal endomorphisms of simple objects in that bimodule category. The odd bimodules in the previously constructed $\mathscr{H}_1-\mathscr{H}_2$ and $\mathscr{H}_1-\mathscr{H}_3$ subfactors give us the list of simple objects in those two categories; note that in each case the three odd bimodules of the same dimension give the same subfactor, since they are just twists by the $\mathbb{Z}/3$ action. Similarly, looking at the even bimodules in those subfactors give the $\mathscr{H}_i-\mathscr{H}_i$ type subfactors. That leaves $\mathscr{H}_2-\mathscr{H}_3$; an $\mathscr{H}_2-\mathscr{H}_3$ subfactor with index $\frac{33+\sqrt{13}}{2} $ can be constructed as in the proof of Lemma \ref{noalg}. It must have principal graphs (6) and (6'), from which the list of simple objects and subfactors may be obtained.    
\end{proof}

\section{The intermediate subfactor lattices}

Another important analogue of ``subgroups'' which appears in subfactor theory is the lattice of intermediate subfactors.  For the $M^G \subset M$ the lattice of intermediate subfactors is the subgroup lattice for $G$.  From the point of view of fusion categories, a subfactor is an algebra object $A$ in a fusion category $\mathscr{C}$, and the intermediate subfactors are just subalgebras of $A$.  This lattice has been studied much more on the subfactor side (e.g. \cite{MR1409040, MR1437496, MR2257402, MR2418197, MR2670925, MR2493615, MR2393428}), while it has not been a major topic in the study of fusion categories.  Nonetheless we believe that this topic should be equally interesting in both fields.  We will use subfactor language in this section so as to be able to use results from \cite{MR2418197} without modification; however none of the results in this section rely seriously on subfactor techniques and could easily be translated into fusion category language.

Recall that the Galois group (written $Gal(M / N)$) of a subfactor $N \subset M$ is the group of automorphisms of $M$ which fix $N$ pointwise.  (In terms of algebra objects, this is just the group of algebra automorphisms of $A$ which restrict to the identity morphism on the unit subobject.)  For an irreducible subfactor, the Galois group is given by the group of invertible objects which are located at a distance of at most $2$ from the identity object (``*'') on the dual principal graph. The invertible objects at depth $2$ are recognizable on the graph as the $1$-valent vertices at depth $2$. The Galois group acts on the lattice of intermediate subfactors $\{ P | N \subseteq P \subseteq M\}$. Since the lattice of intermediate subfactors can be naturally identified with the dual lattice of the intermediate subfactor lattice of the dual subfactor via the basic construction, the Galois group of the dual subfactor also acts on the intermediate subfactor lattice of $N \subset M$. In general these two groups and actions are different.  (From the fusion category perspective, the ``basic construction'' replaces $A \in \mathscr{C}$ with the algebra object $A \otimes A$ in the fusion category of $A$-$A$ bimodule objects in $\mathscr{C}$.)

 We will use the labeling convention $\iota={}_P P_N $ and $\kappa={}_M M_P $ when discussing an intermediate subfactor $N \subset P \subset M$. We will only consider irreducible subfactors. When dealing with multiple intermediate subfactors, we will label the objects $\iota_P, \kappa_P $, etc. If $\alpha$ is an element of the Galois group of $N \subset M$, we will freely identify it with the corresponding $M-M$ bimodule ${}_M L^2(M)_{\alpha(M)}  $.

\begin{lemma} \label{gal}
Let $N \subset P \subset M $ be an intermediate subfactor and let $\alpha$ be a nontrivial element of the Galois group. Then $\alpha(P) \cong P $ as $N-N$ bimodules. If $N \subset P $ and $P \subset M $ have trivial Galois groups, then $\alpha(P) \neq P $.   
\end{lemma}
\begin{proof}
The first statement is obvious. For the second, suppose that $N \subset P $ and $P \subset M $ have trivial Galois groups and $\alpha(P)=P $. By triviality of $Gal(P/N) $, we have $\alpha$ restricted to $P$ is the identity.  Thus $\alpha$ is an element of $Gal(M/P)$ and hence is trivial.
\end{proof}

\begin{lemma} \label{biglattice}
The $\mathscr{H}_2-\mathscr{H}_3$ subfactor $N \subset M$ with graphs (6')-(6') has exactly nine index $4+\sqrt{13}$ intermediate subfactors.
\end{lemma}
\begin{proof}
Let $\xi$ be a noninvertible object in $\mathscr{H}_3 $ and let $\alpha $ be a nontrivial invertible object in $\mathscr{H}_3$, such that $\alpha $ is an element of the Galois group of $N \subset M$. Similarly, let $\rho$ be a noninvertible object in $\mathscr{H}_2 $ and let $\beta $ be a nontrivial invertible object in $\mathscr{H}_3$, such that $\beta $ is an element of the dual Galois group (i.e. the Galois group of the downward basic construction $N_{-1} \subset N \subset M$). 

By the construction in the proof of Lemma \ref{existence} there is at least one intermediate subfactor $N \subset P \subset M$ with index $[M:P]=4+\sqrt{13}$. Without loss of generality we may assume that $[\bar{\iota} \iota] = [1] + [\rho] $. Then by Lemma \ref{gal}, $P, \alpha(P), \alpha^2(P)$ are distinct, and all isomorphic to $[1] + [\rho] $ as $N-N$ bimodules.     

  Let $\bar{P} $ be the dual intermediate subfactor to $P$ in the downward basic construction $N_{-1} \subset N \subset M$. Let $\zeta = {}_N N {}_{\bar{P}} $. Then $[\zeta \bar{\zeta}]=[\bar{\iota}\iota ]=[ 1] + [\rho ] $. Consider the subfactor $\bar{Q}=\beta(\bar{P})$; let $\zeta_Q = {}_N N {}_{\bar{Q}} $. Let $N \subset Q \subset M$ be the dual intermediate subfactor of $\bar{Q}$. Then $[\bar{\iota_Q} \iota_Q]=[\zeta_Q \bar{\zeta_Q}] = [\beta \zeta \bar{\zeta} \beta^{-1} ] =[1]+ [\beta \rho \beta^{-1} ]= [1 ] + [\beta^2 \rho]$. Similarly, the dual subfactor of $\beta^2(\bar{P}) $ is an intermediate subfactor of $N \subset M$ which is isomorphic to $[1] + [\beta \rho] $ as an $N-N$ bimodule. Looking at the actions of $\alpha $ and $\alpha^2 $ on these subfactors shows that there must be three of each.

So we have found nine intermediate subfactors of $N \subset M$ with index $4+\sqrt{13}$. There is a set of three for each $N-N$ bimodule class $[1] + [\rho], [1]+ [\beta \rho ], [ 1] + [\beta^2 \rho] $. The Galois group cyclically permutes each set, and the dual Galois group cyclically permutes the three sets. It remains to show that there are no others.

Once again let $P$ be an intermediate subfactor with $[\bar{\iota} \iota]=[1 ] + [\rho ] $, and let $Q$ be another intermediate with $[\bar{\iota_Q} \iota_Q ]=[1 ]+[\rho ] $. By \cite[Lemma 6.1]{MR1932664} there is an isomorphism $\pi: P \rightarrow Q $ which fixes $N$ pointwise; we will also use $\pi $ to refer to the corresponding $Q-P $ bimodule, $_Q L^2(Q)_{\pi(P)} $. Note that $\pi \iota = \iota_Q $. By a dimension comparison, the $M-M$ sector $[\kappa_Q \pi \bar{\kappa}]$ must contain a sector of dimension $1$, i.e. one of $[1],[\alpha], [\alpha^2]$. Call this one dimensional sector $[\theta] $. Then as in the proof of \cite[Theorem 4.3]{MR2418197}, we have by Frobenius reciprocity $[\theta \kappa]=[\kappa_Q \pi] $, and if we choose $\theta $ in the Galois group then $\theta(P)=Q $. This shows that any intermediate subfactor whose $N-N$ bimodule structure is $[1 ]+[\rho ]$ is in the orbit of $P$ under the action of the Galois group, so there can be only three of them. Similarly, there are only three each for $[1 ]+[\beta \rho ] $ and $[1 ]+[\beta^2 \rho ]$.

\end{proof}

\begin{theorem}
 The subfactors in the list in Theorem \ref{subthm} have the following intermediate subfactor lattices: The subfactor with graphs (5)-(5') has a single proper intermediate subfactor, with index $3$. The subfactor with graphs (2)-(2') has four proper intermediate subfactors, one with index $3$ and the other three with index $\frac{5+\sqrt{13}}{2}$. The subfactor with graphs (6)-(6) has two proper intermediate subfactors, both with index $\frac{5+\sqrt{13}}{2}$. The ($\mathscr{H}_2-\mathscr{H}_3$) subfactor with graphs (6')-(6') has five proper intermediate subfactors: one with index $3$, one with index $\frac{11+3\sqrt{13}}{2}$, and nine with index $4+\sqrt{13}$. All other $\mathscr{H}_i-\mathscr{H}_j$ subfactors have no proper intermediate subfactors.
\end{theorem}

\begin{proof}
 Note that if $N \subset M$ is a $\mathscr{H}_i-\mathscr{H}_j$ subfactor and $N \subset P \subset M$ is a proper intermediate subfactor, then $[M:P],[P:N]$ must be indices of the graphs (1)-(7),(1')-(7'). This immediately implies the last statement.

For the first statement, we see from the graph of (5') that there is a unique index $3$ intermediate subfactor. From the graph (5) we see that there is no co-index $3$ intermediate subfactor. From the list of indices there cannot be any other intermediate subfactors.

In \cite[Theorem 5.19]{MR2418197}, a (2)-(2') subfactor with three index $\frac{5+\sqrt{13}}{2}$ subfactors was constructed. From the graph (2') we see that there is also an index $3$ intermediate subfactor. Let $P$ and $Q$ be distinct index $\frac{5+\sqrt{13}}{2}$ intermediate subfactors. From the graph (2) and \cite[Theorem 3.10]{MR2418197}  we find that the quadrilateral generated by $P$ and $Q$ is noncommuting, and then it follows from \cite[Theorem 5.19]{MR2418197} that $P$ is in the orbit of $Q$ under the action of the Galois group of $N \subset M$ on the intermediate subfactor lattice. Since the Galois group is $\mathbb{Z}/3\mathbb{Z} $, it follows that there are only three intermediate subfactors of index $\frac{5+\sqrt{13}}{2}$. From the graph (2') we see that there is no intermediate subfactor of co-index $\frac{5+\sqrt{13}}{2}$. From the list of indices there cannot be any other intermediate subfactors. 

In \cite[Theorem 5.2]{MR2418197} and the following discussion, a (6)-(6) subfactor $N \subset M$ with two index $\frac{5+\sqrt{13}}{2}$ intermediate subfactors was constructed; by uniqueness it is the same (6)-(6) subfactor which appears on the list in Theorem \ref{subthm}. By \cite[ Lemma 3.14]{MR2418197} there are no other index $\frac{5+\sqrt{13}}{2}$ intermediate subfactors. From the list of indices there cannot be any other intermediate subfactors.


From the graph (6'), we see that the $\mathscr{H}_2-\mathscr{H}_3$ subfactor with graphs (6')-(6') has a unique index $3$ intermediate subfactor and unique co-index $3$ intermediate subfactor. From Lemma \ref{noalg} there is no index $\frac{5+\sqrt{13}}{2}$ intermediate subfactor, and by Lemma \ref{biglattice} there are exactly nine intermediate subfactors with index $4+\sqrt{13}$. From the list of indices there are no other intermediate subfactors.

\end{proof}

\begin{figure}
  
  \begin{center}
    \hpic{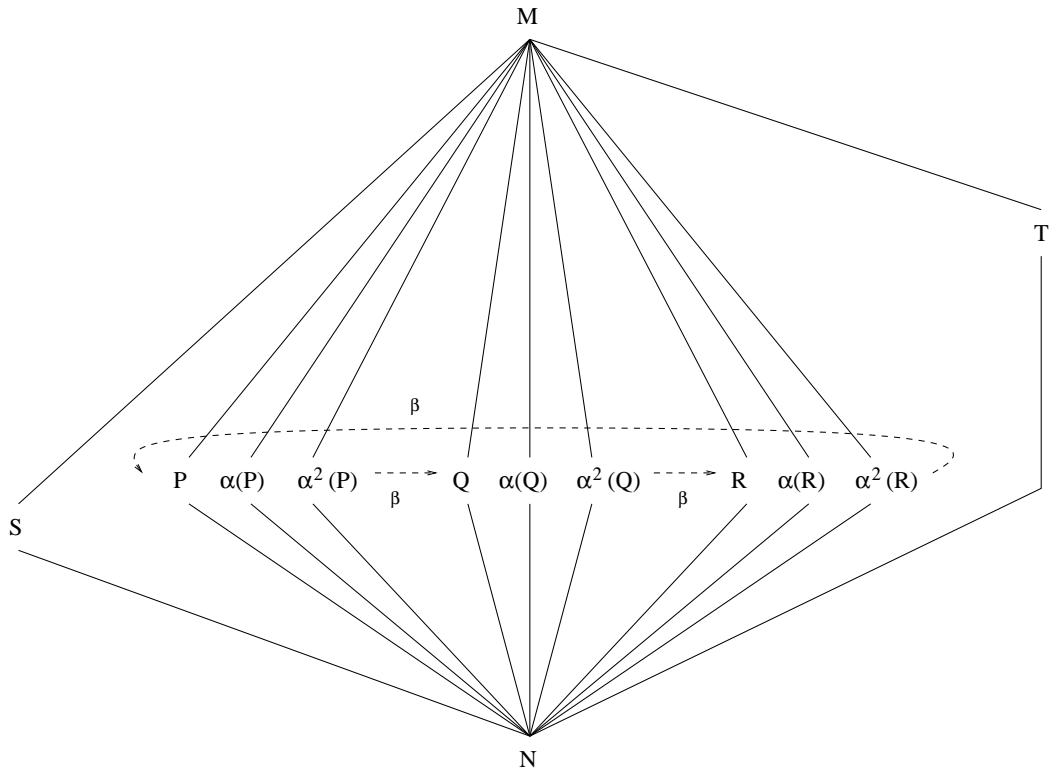} {4in}
  \end{center}
  \caption{The intermediate subfactor lattice of the $\displaystyle \frac{33+9\sqrt{13}}{2} $ subfactor: $G = Gal(M/N)=\{1, \alpha, \alpha^2 \}$, $H=Gal(N/N_{-1})=\{1, \beta, \beta^2 \} $, $S=N \ltimes H  $, $T= M^G $, $[P;N]=[Q:N]=[R:N]=\frac{5+\sqrt{13}}{2} $ .}
\end{figure}

\section{Outer automorphisms and the Brauer-Picard Groupoid}

The goal of this section is to prove that the outer automorphism group of each of the $\mathscr{H}_i$ is trivial.  This completes the calculation of the Brauer-Picard groupoid.  The argument we give uses Emily Peters's description of the Haagerup subfactor planar algebra \cite{0902.1294}.

\begin{lemma}
$\mathrm{Out}(\mathscr{H}_1) \cong \mathrm{Out}(\mathscr{H}_2) \cong \mathrm{Out}(\mathscr{H}_3)$.
\end{lemma}
\begin{proof}
For each $i$ there is only one $\mathscr{H}_i$ module category whose dual is  $\mathscr{H}_i$ (see Lemmas \ref{lem:dualspart1}, \ref{lem:dualspart2}, Definition \ref{lem:dualspart3}, and Theorem \ref{thm:H1H3}).  Thus, the group of Morita autoequivalences of $\mathscr{H}_i$ is isomorphic to the group of outer automorphisms of $\mathscr{H}_i$.   But since all three fusion categories are Morita equivalent, their groups of Morita autoequivalences are all isomorphic to each other.
\end{proof}

Thus, it is enough to show that the outer automorphism group of $\mathscr{H}_1$ is trivial.  We concentrate on $\mathscr{H}_1$ because it is one of the even parts of the Haagerup subfactor (and thus can be described via Peters's construction) and because it has no inner automorphisms.  Essentially the same argument applies directly to $\mathscr{H}_2$ with a little extra care.

In general an automorphism of a fusion category does not give an automorphism of the corresponding planar algebra.  This is because the chosen algebra object is built into the planar algebra formalism.  Only automorphisms which act trivially on the algebra object correspond to automorphisms of the planar algebra.  That is, we must have that $\mathscr{F}(A) \cong A$ and furthermore that the multiplication map is also acted on trivially.  Explicitly this means that the composition $$A \otimes A = \mathscr{A} \otimes \mathscr{A} \rightarrow \mathscr{F}(A \otimes A) \rightarrow \mathscr{F}(A) = A$$ should agree with multiplication, where the first map is part of the data of a tensor functor and the second map is the functor applied to the multiplication morphism.

\begin{lemma}
Any automorphism of $\mathscr{H}_1$ is naturally isomorphic to an automorphism which acts trivially on the algebra object $1+\eta$.
\end{lemma}
\begin{proof}
First note that no other simple object in $\mathscr{H}_1$ has the same dimension as $1$ or $\eta$, hence the automrophism must send $1+\eta$ to itself.  Now the $3$-supertransitivity of the Haagerup subfactor guarantees that up to algebra isomorphism, there is only one algebra structure on $1+\eta$.  Since $1+\eta$ is multiplicity free, we can extend this algebra isomorphism to a natural transformation which sends our original automorphism into an automorphism which acts trivially on $1+\eta$.
\end{proof}

\begin{lemma}
The Haagerup planar algebra (constructed in \cite{0902.1294}) has no nontrivial automorphisms.
\end{lemma}
\begin{proof}
The subspace spanned by the generator $T$ is characterized by being the low weight space for the action of annular Temperley-Lieb on the $4$-box space.  Thus any automorphism of the planar algebra must send $T$ to a multiple of itself.  Since $T^2 = \frac{1}{2} f^{(4)}$ we see that $T$ needs to be sent to $\pm T$.  However, the third twisted moments of $T$ and $-T$ are not equal to each other \cite[Lemma 4.1]{0902.1294}, hence the automorphism must send $T$ to $T$.  Since $T$ generates the planar algebra we see that the automorphism is automatically trivial.
\end{proof}

\begin{remark}
Note that the Haagerup planar algebra does have an anti-linear automorphism interchanging $T$ and $-T$.
\end{remark}

\begin{theorem}
$\mathrm{Aut}(\mathscr{H}_1)$ is trivial.
\end{theorem}
\begin{proof}
By the above lemmas, any automorphism of $\mathscr{H}_1$ is naturally isomorphic to one which acts trivially on the algebra object $1+\eta$ and thus corresponds to an automorphism of the Haagerup planar algebra, which must therefore be trivial.
\end{proof}

\begin{remark}
We have also checked that $\mathrm{Out}(\mathscr{H}_2)$ is trivial in a completely independent way using 6j-symbols instead of planar algebras.  In this other approach $\mathscr{H}_2$ is more convenient because its tensor product rules are multiplicity free, thereby simplifying the theory of 6j-symbols.  First, after possibly applying an inner automorphism, we may assume that the automorphism fixes the object $\xi$.  Furthermore, by looking at the connection (or equivalently the 6j-symbols) any linear automorphism which fixes $\xi$ actually fixes all the objects.  Any such automorphism is a gauge automorphism in the sense of \cite{Liptrap}.  All gauge automorphisms can be found following \cite{Liptrap} using  nothing more than highschool algebra.  However, our argument, elementary as it was, was also extremely tedious.  Since we were unable to find a good way to shorten or clarify this argument we have chosen not to inflict it on the reader.
\end{remark}

\section{The Izumi subfactors}
In \cite[Section 7]{MR1832764} Izumi listed a system of equations associated to a finite Abelian group $G$ of odd order such that any solution gives a unitary fusion category with sectors $[\alpha_{g}], [\alpha_{g} \xi], g \in G$ satisfying $[\alpha_g \xi ]= [\xi \alpha_{-g}] $ and $[\xi^2 ]=\sum_{g \in G} \limits [\alpha_g \xi ] $. He showed that $[1] + [\xi ] $ admits a Q-system.  The sector $[\xi]$ has Forbenius-Perron dimension $d=d(\xi)=\displaystyle \frac{n+\sqrt{n^2+4}}{2}$, where $n=|G|$, so the corresponding subfactor has index $1+d$. Note the relation $d^2=1+nd$. For $\mathbb{Z} / 3 \mathbb{Z}$ one recovers the Haagerup subfactor, and he showed that there is unique solution up to equivalence for $\mathbb{Z} / 5 \mathbb{Z} $.   The goal of this section is to discuss how our results generalize to other Izumi subfactors.

Izumi's equations were solved for $\mathbb{Z}/n\mathbb{Z} $ for $n=7,9$ (in the latter case there are two non-equivalent solutions) by Evans and Gannon \cite[Theorem 5]{1006.1326}, who also found numerical evidence for solutions for the cases $n=11,13,15,17,19$. They also computed the fusion rings of the dual fusion categories of the subfactors subject to conditions on the modular data \cite[Theorem 7]{1006.1326}; these conditions are satisfied for at least one solution for each $n$ for which solutions are known. 

The dual fusion ring is then commutative and is described as follows: there is the identity $1$; there is a simple object $\eta$ with $d(\eta)=d $; there are $\frac{n-1}{2} $ simple objects $\nu_j$ with $d(\nu_j)=d+1$; and there are $\frac{n-1}{2} $ simple objects $\mu_j$ with $d(\mu_j)=d-1$.

From now on we will fix $n$ and let $\mathscr{I}_1 $ and $\mathscr{I}_2$ denote, respectively, the commutative and noncommutative fusion categories of an Izumi subfactor associated to $\mathbb{Z} / n \mathbb{Z} $ and satisfying the Evans-Gannon conditions. Let $I_1 $ be the fusion ring of $\mathscr{I}_1  $ and let $I_2 $ be the fusion ring of $\mathscr{I}_2 $. Let $\alpha=\alpha_g$ be an object corresponding to a generator of the group, so that the invertible objects are given by $1,\alpha^i, 1 \leq i \leq n-1$.
 
If $\gamma$ is a simple algebra object in a fusion category, then we have $(\gamma,\rho) \leq d(\rho ) $ for every simple object $\rho $ in $\mathscr{C} $.

\begin{definition}
 A simple algebra object $\gamma$ in a fusion category $\mathscr{C}$ will be called saturated if $(\gamma,\rho )=\lfloor d(\rho) \rfloor $ for every simple object $\rho$ in $\mathscr{C}$. 
\end{definition}

If $[\gamma]$ is a saturated simple algebra object in $\mathscr{C} $, then the vertex in the principal graph corresponding to the fundamental object is connected to the vertex corresponding to $\rho$ by $\lfloor d(\rho) \rfloor$ edges, for every simple object $\rho$ in $\mathscr{C}$. Let $G_{\mathscr{C}}$ be the graph obtained by adding one more odd vertex, along with one edge from the new vertex to each non-invertible simple object in $\mathscr{C}$. (For examples, see Example \ref{ex}, (c), and Haagerup graph (6')).

\begin{lemma}
 Let $\mathscr{C}$ be a fusion category with fusion ring $I_1$ or $I_2$. Then the only possible principal graph of a saturated simple algebra is $G_{\mathscr{C}} $. In either case, the Frobenius-Perron weight of the second odd vertex is $\sqrt{n}$ and the dimension of the algebra object is $n+n^2d$.  
\end{lemma}
\begin{proof}
 For $I_2$, one can check using the Evans-Gannon fusion rules that all entries of the reduced fusion matrix of the object $\sum_{i=0}^{n} \limits  \lfloor d(\rho) \rfloor \rho_i $ are $1$; the Frobenius-Perron weights are easy to compute. For $I_1$, the proof is the same except that only the entries corresponding to two non-invertible objects are $1$. 
\end{proof}

\begin{corollary} \label{imp}
 Let $\mathscr{C}$ be a fusion category with fusion ring $I_2$. Then $\mathscr{C} $ does not have a saturated algebra object.
\end{corollary}

\begin{proof}
 As in Example \ref{ex}, (c), let $\kappa'$ be the object corresponding to the second odd vertex. Then $[\kappa' \bar{\kappa'}]=[1]+[ \sigma] $, where $[\sigma] $ is a nonnegative integral linear combination of $[\nu_j], [\eta], [\mu_j] $. Since $d(\sigma )=n-1$, this is impossible.
\end{proof}


\begin{theorem} \label{I3}
There exists a unitary fusion category $\mathscr{I}_3 $ which is Morita equivalent to $\mathscr{I}_1$ or $\mathscr{I}_2$ but not isomorphic to either of them.
\end{theorem}
\begin{proof}
 Let $\lambda$ denote the simple object corresponding to the middle vertex of the dual graph of the Izumi subfactor. Let $\gamma=\lambda \bar{\lambda}$ be the correponding algebra object, which is also a Q-system. Then $\gamma$ is the $M-M$ Q-system for an $\mathscr{I}_1-\mathscr{I}_2$ subfactor $N \subset M$ with index $d(\gamma)=\frac{(nd)^2}{d+1}$. For any $0 \leq k \leq n$, $(\gamma, \alpha^k)=(\lambda, \alpha^k \lambda )=(\lambda, \lambda)=1 $, so there is an intermediate $\mathscr{I}_1-*$ subfactor $N \subset P \subset M$ with index $[P:N]=d(\gamma )/n = \frac{nd^2}{d+1}=1+(n-1)d$. Let $\mathscr{I}_3$ be the $P-P$ even part of $N \subset P$.

Then $\mathscr{I}_3 $ contains a nontrivial invertible object, so $\mathscr{I}_3 \ncong \mathscr{I}_1$. Suppose $\mathscr{I}_3 \cong \mathscr{I}_2$. Then as in the proof of Lemma \ref{noalg}, there would have to exist an $\mathscr{I}_1-* $ subfactor $R \subset S$ of index $(1+d)(1+(n-1)d )=n(1+nd)$. But this would imply the existence of a saturated algebra object in $\mathscr{I}_1$, which is impossible by Lemma \ref{imp}. 
\end{proof}

As we have seen, in the case $n=3$, $\mathscr{I}_3$ has fusion ring $I_2$ and there is a unique simple bimodule category between each pair $\mathscr{I}_i, \mathscr{I}_j, 1 \leq i,j \leq 3 $ up to isomorphism. It is natural to wonder whether this holds true for other values of $n$. The first question that needs to be resolved is finding the principal graphs of the $\mathscr{I}_1-\mathscr{I}_3$ subfactor constructed in Theorem \ref{I3}. There are natural candidates for these graphs based on the $n=3$ and $n=5$ (below) cases, but we cannot verify the following conjecture for arbitrary $n$ at this time.

\begin{conjecture}
 Let $N \subset M$ be the  $\mathscr{I}_1-\mathscr{I}_3$ subfactor constructed above, and let $\gamma$ be the corresponding algebra object in $\mathscr{I}_1 $. Then $\gamma \cong 1 + \sum_{j=1}^{\frac{n-1}{2}} \limits (\nu_j + \mu_j) $.

%
%
\end{conjecture}

\subsection{The Izumi subfactor for $\mathbb{Z} / 5 \mathbb{Z}$}

There is a unique Izumi subfactor corresponding to $n=5$. We consider a list of graphs analogous to that obtained in the Haagerup ($n=3$) case. As before, the unprimed series are consistent with the $I_1$ fusion rules and the primed series is consistent with the $I_2 $ fusion rules.

(1) \hpic{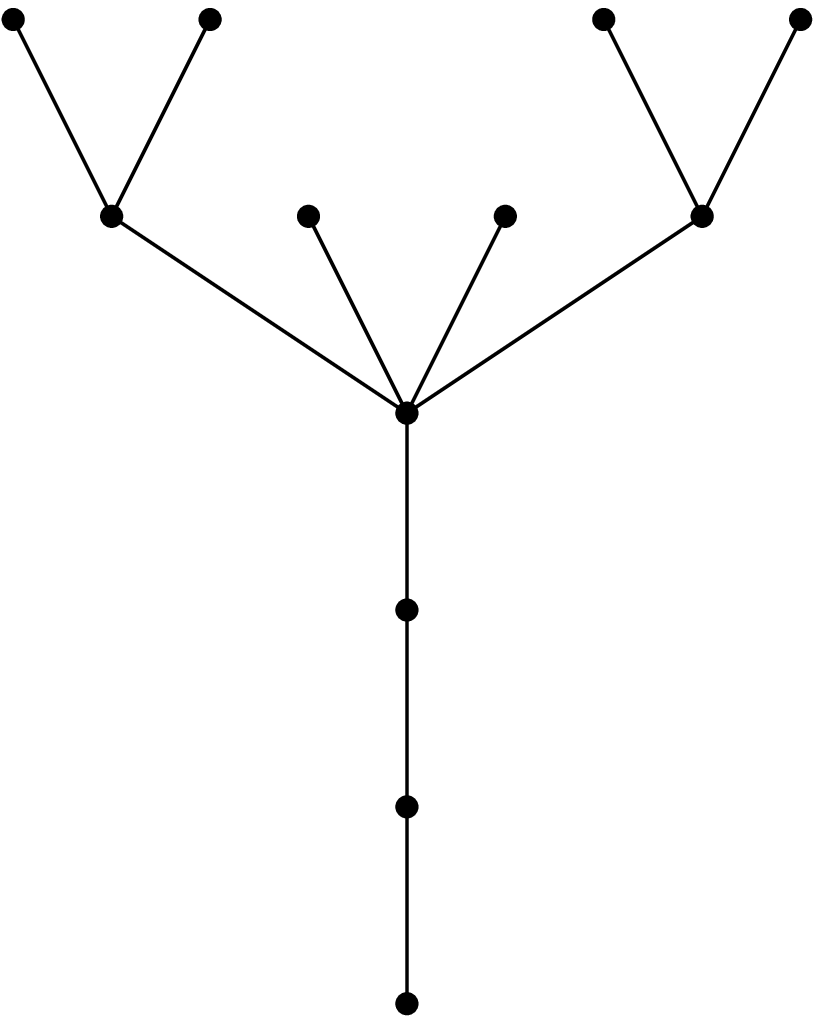} {1.2in} , $\displaystyle \frac{7+\sqrt{29}} {2}$, (2) \hpic{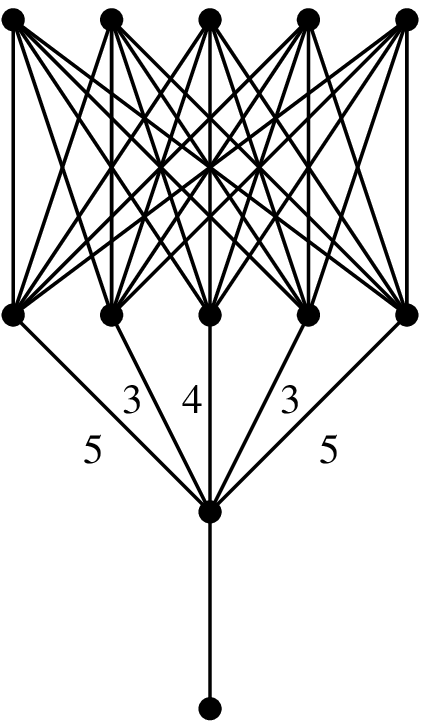} {1.2in} , $55+10\sqrt{29}$ \\

(3) \hpic{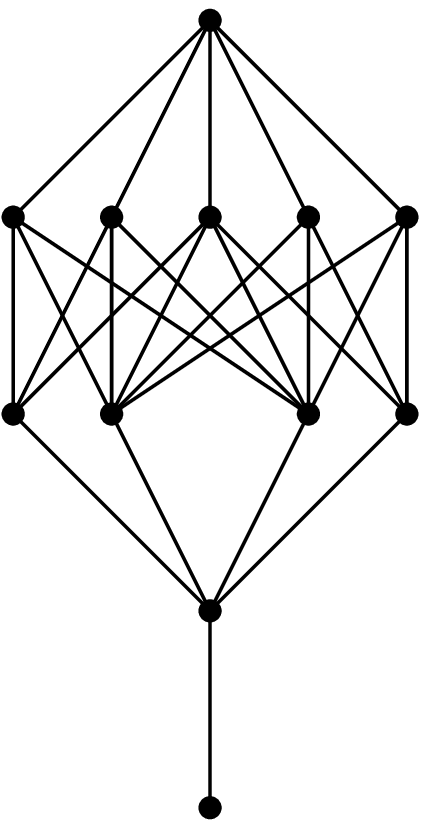} {1.2in} , $11+2\sqrt{29}$, (4) \hpic{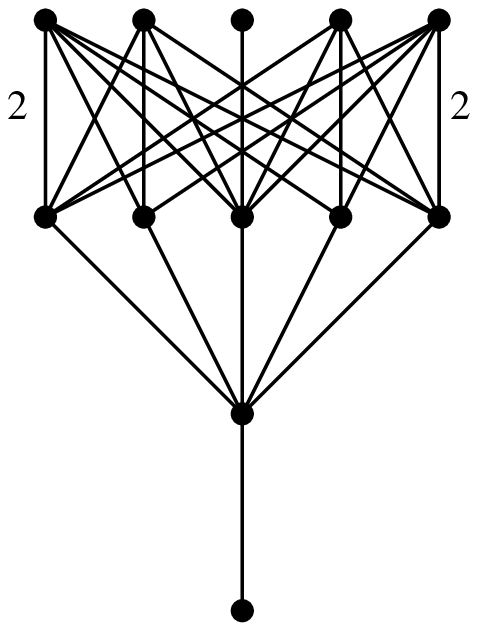} {1.4in} , $\displaystyle \frac{27+5\sqrt{29}} {2}$ \\

(5) \hpic{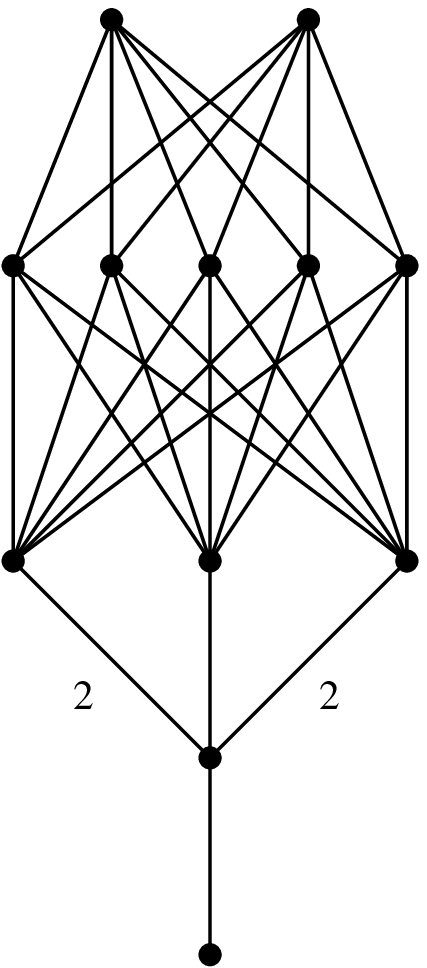} {1.2in} , $\displaystyle \frac{35+5\sqrt{29}} {2}$, (6) \hpic{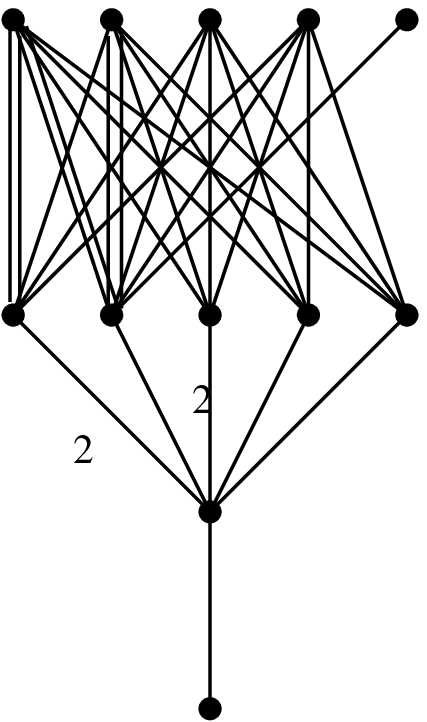} {1.2in} , $\displaystyle \frac{39+7\sqrt{29}} {2}$ \\

(7) \hpic{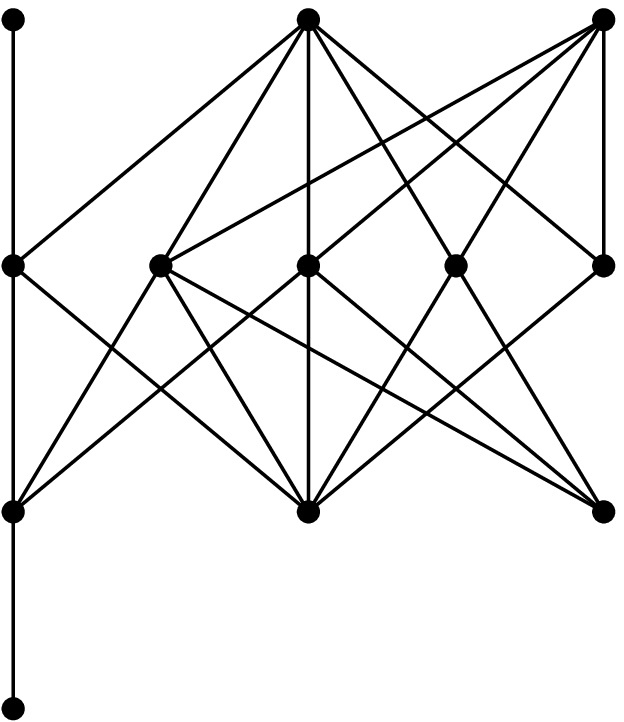} {1.2in} , $\displaystyle \frac{19+3\sqrt{29}} {2}$ \\

(1') \hpic{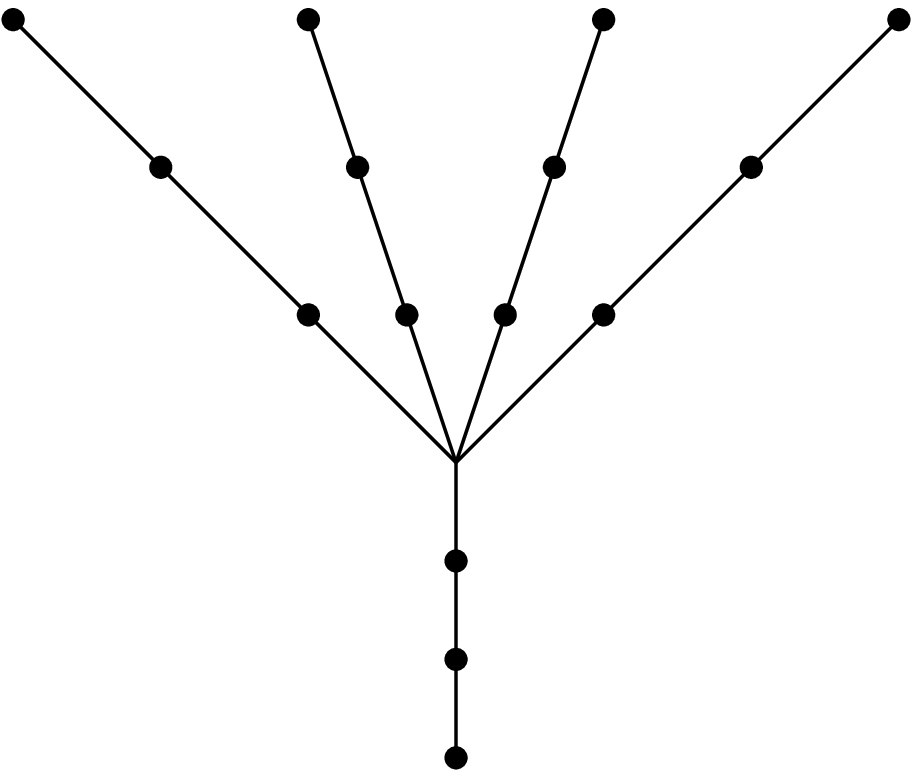} {1.2in} , $\displaystyle \frac{7+\sqrt{29}} {2}$, (2') \hpic{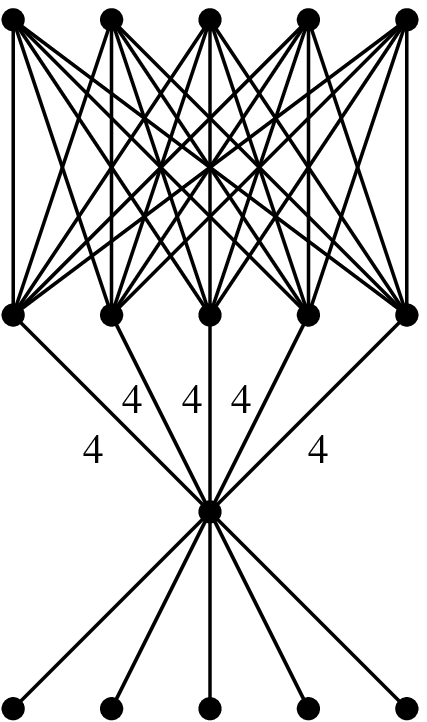} {1.2in} , $55+10\sqrt{29}$ \\

(3') \hpic{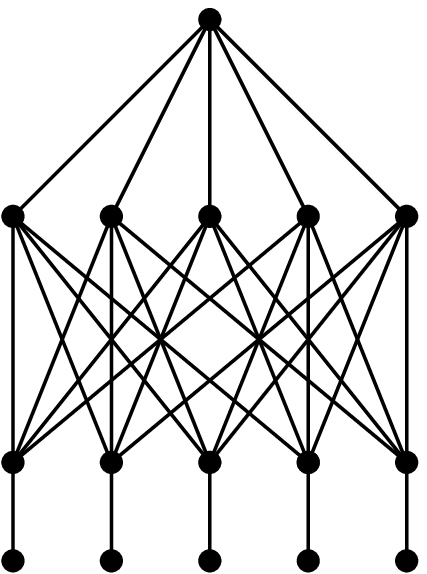} {1.2in} , $11+2\sqrt{29}$, (4') \hpic{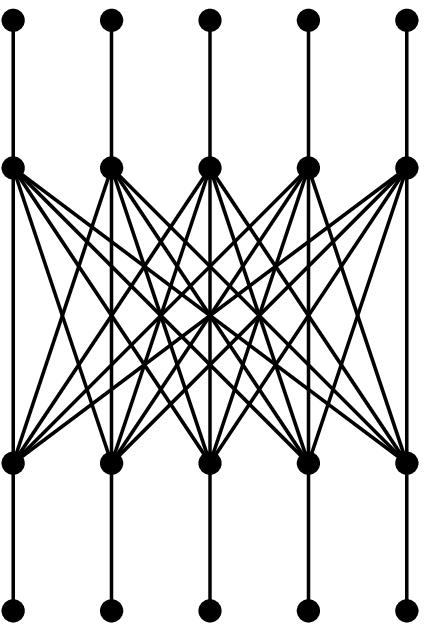} {1.2in} , $\displaystyle \frac{27+5\sqrt{29}} {2}$ \\

(5') \hpic{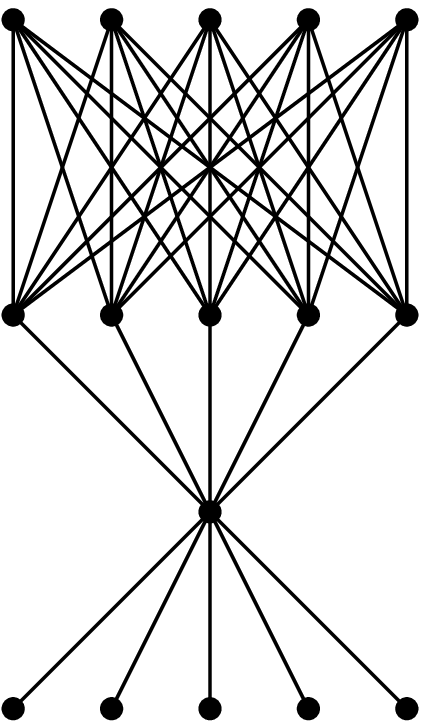} {1.2in} , $\displaystyle \frac{35+5\sqrt{29}} {2}$, (6') \hpic{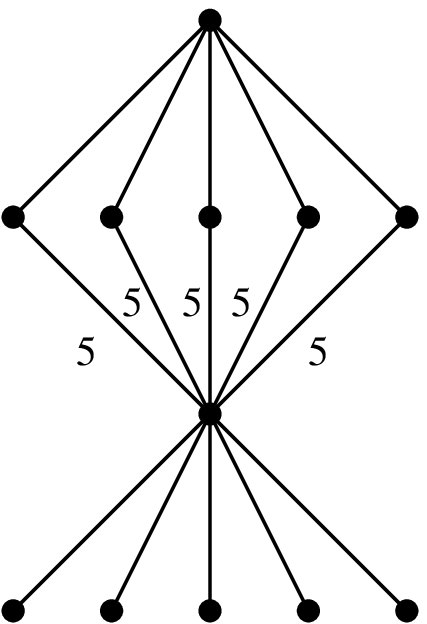} {1.2in} , $\displaystyle \frac{135+25\sqrt{29}} {2}$ \\

(7') \hpic{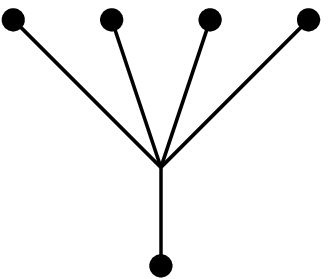} {0.8in} , $5$ \\

\begin{lemma}\label{i10graphs}
 Let $[\gamma] $ be a nontrivial simple algebra object in a fusion category with fusion ring isomorphic to $I_2$. Then the principal graph of $[\gamma]$ is one of the seven graphs (1')-(7') listed above.
\end{lemma}

\begin{proof}
 Besides (1')-(7'), there are six other irreducible graphs that are compatible with the $I_2$ fusion rules. These are: the graphs obtained from (2') and (5') by in each case replacing four of the five symmetric odd vertices by a single odd vertex with two edges to each of the five even vertices; and two pairs of graphs each at the indices $\displaystyle \frac{85+15\sqrt{29}} {2}$ and $30+5\sqrt{29}$.  In each case, the square of the Frobenius-Perron dimension of one of the other odd vertices gives an index which does not admit a graph consistent with $I_2$.
\end{proof}

The proofs of the following results are now essentially the same as in the $n=3$ case.

\begin{lemma}
 There exist unique algebra structures in $\mathscr{I}_2 $ on $1$, $1+\xi$, $1+\alpha+\alpha^2+\alpha^3+\alpha^4$; there does not exist an algebra structure on $1+\alpha^k \xi+\alpha^{k+1} \xi+\alpha^{k+2} \xi + \alpha^{k+3} \xi$ for any $k$.
\end{lemma}

\begin{lemma}
 The $\mathscr{I}_1-\mathscr{I}_3 $ subfactor constructed in \ref{I3} has principal and dual graphs (3) and (3').
\end{lemma}

\begin{lemma}
 The fusion category $\mathscr{I}_3 $ has fusion ring $I_2$.
\end{lemma}

\begin{theorem}
Each $\mathscr{I}_i, i=1,2,3 $ has exactly three module categories, whose dual categories in each case are again the three $\mathscr{I}_i$. The complete list of subfactors whose fusion categories are Morita equivalent to $\mathscr{I}_i$ is as follows; each exists and is unique up to isomorphism of the planar algebra. The $\mathscr{I}_i-\mathscr{I}_i $ subfactors are self-dual while the $\mathscr{I}_i-\mathscr{I}_j $ subfactors for $i \neq j $ come in non-self-dual pairs; we only list each pair once.

(a) $\mathscr{I}_1-\mathscr{I}_2$: One with principal graphs (1)-(1') and one with principal graphs (2)-(2').

(b) $\mathscr{I}_1-\mathscr{I}_3$: One subfactor with principal graphs (3)-(3') and one with principal graphs (5)-(5').

(c) $\mathscr{I}_2-\mathscr{I}_3$: One subfactor with principal graphs (6')-(6'). 

(d) $\mathscr{I}_1-\mathscr{I}_1$: One subfactor each with principal graphs (4)-(4), (6)-(6), and (7)-(7).

(e) $\mathscr{I}_2-\mathscr{I}_2$: One subfactor with principal graph (4')-(4').

(f) $\mathscr{I}_3-\mathscr{I}_3$: One subfactor with principal graph (4')-(4').
\end{theorem}

\newcommand{\urlprefix}{}
\bibliographystyle{alpha}
\bibliography{bibliography}

\end{document}